\documentclass[12pt]{amsart}

\usepackage{amsmath, amsthm}
\usepackage{graphicx,epsfig}
\usepackage{amsmath,latexsym,amssymb,verbatim}

\newtheorem{definition}{Definition}
\newtheorem{lemma}{Lemma}
\newtheorem{proposition}{Proposition}
\newtheorem{theorem}{Theorem}
\newtheorem{corollary}{Corollary}
\newtheorem{remark}{Remark}
\newtheorem{example}{Example}

\def\sideremark#1{\ifvmode\leavevmode\fi\vadjust{\vbox to0pt{\vss % the remark
      \hbox to 0pt{\hskip\hsize\hskip1em           %                will appear only
 \vbox{\hsize2cm\tiny\raggedright\pretolerance10000%                on the side
 \noindent #1\hfill}\hss}\vbox to8pt{\vfil}\vss}}}%
\newcommand{\R}{\mathbb{R}}

\newcommand{\N}{\mathbb{N}}

\newcommand{\I}{\mathcal{I}}

\newcommand{\Diff}{\text{Diff\ }}

\begin{document}
\title[Multiplicity and $\varepsilon$- neighborhoods]{Multiplicity  of fixed points\\ and \\  growth of 
$\varepsilon$-neighborhoods of orbits }
\author{Pavao Marde\v si\' c$^1$, Maja Resman$^2$, Vesna \v Zupanovi\' c$^2$}
\date{}
\let\thefootnote\relax\footnotetext{This article was supported by the Franco-Croatian PHC-COGITO project 24710UJ M}

%\end{document}
 \address{$^1$Universit\' e de Bourgogne, Department de Math\' ematiques, Institut de Math\' ematiques de Bourgogne, B.P. 47 870-21078-Dijon Cedex, France}
 \address{$^2$University of Zagreb, Department of Applied Mathematics, Faculty of Electrical Engineering and Computing, Unska 3, 10000 Zagreb, Croatia}
\email{mardesic@u-bourgogne.fr, 
maja.resman@fer.hr,\newline vesna.zupanovic@fer.hr}
\begin{abstract}
We study the relationship between the multiplicity of a fixed point of a function $g$, and the dependence on $\varepsilon$ of  the length of $\varepsilon$-neighborhood of any orbit of $g$, tending to the fixed point.
The relationship between these two notions was discovered in Elezovi\' c, \v Zubrini\' c, \v Zupanovi\' c \cite{elezovic} in the differentiable case, and related to the box dimension of the orbit. 

Here, we generalize these results to non-differentiable cases introducing a new notion of critical Minkowski order. 
We study the space of functions having a development in a Chebyshev scale and use multiplicity with respect to this space of functions.  
With the new definition, we recover the relationship between multiplicity of fixed points and the dependence on $\varepsilon$ of the length of $\varepsilon$-neighborhoods of orbits in non-differentiable cases. 

Applications include in particular Poincar\' e maps near homoclinic loops and hyperbolic 2-cycles, and Abelian integrals. 
This is a new approach to estimate the cyclicity, by computing the length of the $\varepsilon$-neighborhood of one orbit of the Poincar\' e map (for example numerically), and by comparing it to the appropriate scale.
\end{abstract}

\maketitle



\bigskip

\section{Introduction}\label{sec1}%\edz{Promjena!}

The multiplicity of a fixed point of a differentiable function can be seen from the density of its orbit near the fixed point as was shown in \cite{elezovic}. We recall this result in Theorem~\ref{neveda}. The information on the density is contained in the behavior of $\varepsilon-$neighborhood of the orbit near the fixed point and is usually measured by the box dimension of the orbit. It was further noted in \cite{belg} that the box dimension of the orbit of Poincar\' e map around a focus or a limit cycle shows how many limit cycles can appear in bifurcations. This gave an application of the result from Theorem~\ref{neveda} to continuous dynamical systems.  

The idea of this article is to generalize these results to a class of functions which are non-differentiable at a fixed point. The goal is again to estimate the multiplicity of a fixed point of such a function only from the asymptotic behavior of the length of the $\varepsilon$-neighborhood of any of its orbits close to the fixed point,  as $\varepsilon\to 0$. The results can be applied to continuous dynamical systems. While differentiable functions described above appear as displacement functions near limit cycles and foci, non-differentiable functions appear naturally as displacement functions near polycycles, see e.g. \cite{IY}, \cite{roussarie} (see Section \ref{sec4}). The multiplicity of a fixed point 0 of the displacement function near some limit periodic set reveals the number of limit cycles that appear in the unfoldings of the limit periodic set. It is of interest to find at least an upper bound on the multiplicity. 

Calculating numerically an orbit of the Poincar\' e map of the limit periodic set, one can estimate the length of its $\varepsilon$-neighborhood for small values of $\varepsilon$ and thus estimate its asymptotic behavior. 

In the differentiable case (foci, limit cycles), see Theorem~\ref{neveda}, it suffices to compare the behavior of the length with discrete scale of powers, $\{\varepsilon,\ \varepsilon^{1/2}, \varepsilon^{1/3},\ldots\}$. The moment when comparability occurs reveals multiplicity, i.e. cyclicity. We see additionally that this moment is signaled by the limit capacity (the box dimension) of the orbit which actually shows the density of the orbit around the fixed point: the bigger this density is, more limit cycles can appear in perturbations.

In non-differentiable cases, however, we show that it is not sufficient to compare the length of the $\varepsilon$-neighborhood with the scale of powers to estimate multiplicity and cyclicity. The idea behind Theorems~\ref{gensaus}~and~\ref{chebsaus} is in finding the appropriate scale to which the length should be compared to obtain precise information on the multiplicity. This scale, as we will see, depends on the unfolding and should be estimated at least from above. Here, instead of box dimension, the new notion of critical Minkowski order is introduced, to signal the moment when  the comparability occurs in the new scale.

\vspace{0.2cm}
The article is organized as follows.
First, in Subsection 1.1 we recall the connection from \cite{elezovic} between the box dimension of the orbit and the multiplicity of the fixed point in the differentiable case, see Theorem~\ref{neveda}. In Subsection 1.2 we recall and introduce definitions and notions we need in non-differentiable cases. Finally, in Section~\ref{sec2}, we state our main results concerning non-differentiable cases, see Theorem~\ref{gensaus} and Theorem~\ref{chebsaus}. Some applications to continuous dynamical systems are given in Section~\ref{sec4}.

\subsection{Differentiable case.}\

Denote ${{ {\Diff}}}^r[0,d)$ the space of  $C^r$-differentiable functions on $[0,d)$, for $r$ sufficiently big, $d>0$. Let  $f\in \Diff^r[0,d)$, $f(0)=0$ and\linebreak $x>f(x)>0$, 
\begin{comment}
\end{comment}
for $x\in(0,d)$. Put
\begin{equation}\label{g}g=id-f
\end{equation}
and consider the orbit $S^g(x_0)$ of $0<x_0<d$ by $g$:
\begin{equation}\label{orbit}
S^g(x_0)=\{x_n\ |\ n\in\mathbb{N}\},  \quad x_{n+1}=g(x_n). 
\end{equation}

Let $\mu_0^{fix}(g)$  be \emph{the multiplicity of $0$ as a fixed point} of the function $g$ in the family ${{ {\Diff}}}^r[0,d)$. That is, the number of fixed points that can bifurcate from $0$ by bifurcations within ${{ {\Diff}}}^r[0,d)$. Then, 
\begin{equation}\label{mul}
\mu_0^{fix}(g)=k, \quad \text{if} \quad f(0)=f'(0)=\ldots=f^{(k-1)}(0)=0,\ f^{(k)}(0)\neq 0,
\end{equation}
i.e., $0$ is a zero of multiplicity $\mu_0(f)=k$ of $f$.

%In this differentiable case Elezovi\' c, \v Zubrini\' c, \v Zupanovi\' c in \cite{elezovic} have found an explicit formula that connects multiplicity of a zero as a fixed point of the function $g$ and the length of the $\varepsilon-$neighborhood (as well as the box dimension) of any orbit of the function $g$.

Now we define the Minkowski content and the box dimension of a bounded set. Let $U\subset\mathbb{R}^N$ be a bounded set. Denote by $|A_{\varepsilon}(U)|$ the Lebesgue measure of  $\varepsilon-$neighborhood of $U$.

By {\it lower and upper $s$-dimensional  Minkowski content of $U$}, $s\ge0$, we mean

$$
{\mathcal M}_{*}^{s} (U)=\liminf_{{\varepsilon}\to0}\frac{|A_{\varepsilon}(U)|}{{\varepsilon}^{N-s}}\text{\ \  and\ \ }{\mathcal M}^{*s} (U)=\limsup_{{\varepsilon}\to0}\frac{|A_{\varepsilon}(U)|}{{\varepsilon}^{N-s}}
$$
respectively. Furthermore, \emph{lower and upper box dimension 
%(also called lower and upper limit capacity)\edz{NOVO!!!}
 of $U$} 
are defined by
$$
\underline{\dim}_B U=\inf\{s\ge0\  |\ {\mathcal M}_{*}^s(U)=0\},\ \overline{\dim}_B U=\inf\{s\ge0\  |\ {\mathcal M}^{*s}(U)=0\}.
$$
As functions of $s\in[0,N]$, $\mathcal{M}^{*s}(U)$ and $\mathcal{M}_*^s(U)$ are step functions that jump only once from $+\infty$ to zero as $s$ grows, and upper and lower box dimension contain information on jump in upper and lower content respectively.\\
If $\underline{\dim}_B U=\overline{\dim}_B U$, then we put $\dim_B(U)=\underline{\dim}_B U=\overline{\dim}_B U$ and call it the \emph{box dimension of $U$}. In the literature, upper box dimension (also called \emph{limit capacity})
has been widely used.
For more details on box dimension, see Falconer \cite{falconer} or Tricot \cite{tricot}.

\medskip

Our case is  $1$-dimensional  so $N=1$ in the definition of box dimension, and also $U\subset [0,d)$, where $d>0$. We are interested in measuring the density of accumulation of the orbit of a function $g$ near its fixed point zero.
Let $g$ be sufficiently differentiable on $[0,d)$, $d>0$, such that $g(0)=0$, 
we denote by $S^g(x_0)$, $0<x_0<d$, the orbit of $x_0$ by $g$
 defined by $x_{n+1}=g(x_n)$, $x_0<d$, and tending monotonously to zero. 
 In $1$-dimensional differentiable case it is verified that $\dim_B (S^g(x_0))$ is independent of the choice of the point $x_0$ in the basin of $0$. Therefore one can define
\emph{box dimension of a function $g$}  by
$$
\dim_B (g)=\dim_B (S^g(x_0)),
$$
for any $x_0$ from the basin of attraction of $0$.

\smallskip
For two positive functions $F(x)$ and $G(x)$, with no accumulation of zeros at $x=0$, we write $F(x)\simeq G(x)$, as $x\to 0$, if there exist two positive constants $A$ and $B$ and a constant $d>0$ such that $A\leq F(x)/G(x)\leq B$, $x\in(0,d)$, and call such functions \emph {comparable}.  
In the sequel, we write $F(x)=o(x)$, if $\lim_{x\to 0}\frac{F(x)}{x}=0$.
\medskip 

Now we reformulate Theorem 1 from \cite{elezovic}, connecting box dimension and multiplicity in the differentiable case.

\begin{theorem}\label{neveda}

Let $f$ be sufficiently differentiable on $[0,d)$ and positive and strictly increasing on $(0,d)$.
Let $g=$id$-f$ and 
suppose that the multiplicity of $0$ as a fixed point of $g$ is finite and greater than 1. That is, $1<\mu_0^{fix}(g)<\infty$. Let $x_0\in(0,d)$,  $S^{g}(x_0)$ be defined as in \eqref{orbit} and let $|A_\varepsilon(S^g(x_0))|$ be the length of the 
$\varepsilon$-neighborhood  of the orbit $S^g(x_0)$, $\varepsilon>0$. 

Then
\begin{equation}\label{sausage}
|A_\varepsilon(S^g(x_0))| \simeq \varepsilon^{1/\ \mu_0^{fix}(g)},\text{ as $\varepsilon\to 0$}.
\end{equation} 

If $\mu_0^{fix}(g)=1$ and additionaly $f(x)<x$ on $(0,d)$, then
\begin{equation}\label{mu1}
|A_\varepsilon(S^g(x_0))| \simeq \left\{ \begin{array}{ll} \varepsilon(-\log \varepsilon),&\text{ if }f'(0)<1\\
\varepsilon \log(-\log \varepsilon),&\text{ if }f'(0)=1
\end{array}\right.,\text{ as $\varepsilon\to 0$}.
\end{equation} 

Moreover, for $1\leq \mu_0^{fix}(g)<\infty$,
\begin{equation}\label{boxd}
\mu^{fix}_0 (g)=\frac1{1-\dim_B(g)}.
\end{equation}
\end{theorem}

\medskip
\noindent\emph{Sketch of proof.} We illustrate the proof on the simplest case when $g$ is linear, $g(x)=kx,\ k\in(0,1)$. Take any initial point $x_0\in(0,d)$. By recursion, it is easy to compute the whole orbit by $g$: \begin{equation}\label{orbi}x_n=k^n x_0,\ n\in\mathbb{N}.\end{equation}

To compute the asymptotic behavior of the length of the\linebreak $\varepsilon$-neighborhood of the orbit, we divide the $\varepsilon$-neighborhood in two parts: the nucleus, $N_\varepsilon$, and the tail, $T_\varepsilon$. The tail is the union of all disjoint $(2\varepsilon)$-intervals of the $\varepsilon$-neighborhood, before they start to overlap. It holds that
\begin{equation}\label{ukku}
|A_\varepsilon(S^g(x_0))|=|N_\varepsilon|+|T_\varepsilon|.
\end{equation}

Let $n_\varepsilon$ denote the index separating the tail and the nucleus. It describes the moment when $(2\varepsilon)$-intervals around the points start to overlap. 
We have that
\begin{equation}\label{up1}
|N_\varepsilon|=x_{n_\varepsilon}+\varepsilon,\quad |T_\varepsilon|\simeq n_\varepsilon\cdot\varepsilon,\ \varepsilon\to 0.
\end{equation}
To find the asymptotics of $n_\varepsilon$, we have to solve $x_{n_\varepsilon+1}-x_{n_\varepsilon}\simeq 2\varepsilon$, $\varepsilon\to 0$. Using \eqref{orbi}, we get
$$
x_{n_\varepsilon}\simeq \varepsilon,\ n_\varepsilon\simeq -\log\varepsilon,\ \varepsilon\to 0.
$$
From \eqref{up1}, we get
$$
|N_\varepsilon|\simeq \varepsilon,\ |T_\varepsilon|\simeq \varepsilon(-\log\varepsilon),
$$
therefore, by \eqref{ukku}, $|A_\varepsilon(S^g(x_0))|\simeq \varepsilon(-\log\varepsilon)$, as $\varepsilon\to 0$.\qed

\medskip
Note that in Theorem~\ref{neveda} we assume $f$ to be differentiable at zero point $x=0$. In this article, we generalize Theorem~\ref{neveda} to some non-differentiable functions at $x=0$.  

Since in the non-differentiable cases standard multiplicity of zero is not well defined, we use the notion of multiplicity of a point as zero of $f(x)$ with respect to a family of functions (see Definition \ref{mult}). The family of functions we consider will be the family of functions having a finite codimension asymptotic development with respect to a Chebyshev scale
(see Definition \ref{cheb}). 

As the main results, in Section \ref{sec2} we extend formula $(\ref{sausage})$ to non-differentiable case. Therefore we have to introduce the notion of \emph{critical Minkowski order with respect to a Chebyshev scale}, see Definition~\ref{generaldim}. This notion is in the differentiable case directly related to box dimension, see Remark~\ref{gbd}.$ii)$. In non-differentiable cases, it generalizes the notion of box dimension in a way that a formula similar to formula $(\ref{boxd})$ holds. 
\medskip

\subsection{Non-differentiable cases.}\

Let us recall some definitions we use in the non-differentiable cases.
\begin{definition}\label{mult}
Let $\Lambda$ be a topological space and let $\mathcal{F}=\{f_\lambda|\ \lambda\in \Lambda\}$, $f_\lambda: [0,d)\rightarrow \mathbb{R}$, be a family of functions. Let $\lambda_0\in\Lambda$,
we say that $x=0$ is \emph {a zero of multiplicity greater than or equal to $m$ of the function $f_{\lambda_0}$ in the family of functions ${\mathcal F}$}
%\emph {a zero of multiplicity greater than or equal to $m$ of the function $f_{\lambda_0}$ within the family $\{f_\lambda\}$}
 if there exists a sequence of parameters $\lambda_n\to\lambda_0$, as $n\to\infty$, such that, for every $n\in\mathbb{N}$, $f_{\lambda_n}$ has $m$ distinct zeros $y_1^n,\ldots,y_m^n\in[0,d)$ different from $x=0$ and $y_j^n\to 0$, as $n\to\infty$, $j=1,\ldots,m$.\\
We say that $x=0$ is a zero of \emph{of multiplicity $m$ of the function $f_{\lambda_0}$ in the family ${\mathcal F}$}
and write 
$\mu_0 (f_{\lambda_0},{\mathcal F})=m$,
if $m$ is the biggest possible integer such that the former holds. 

Putting $g_\lambda=id-f_\lambda$, $\mathcal{G}=\{g_\lambda=id-f_\lambda\}$, the multiplicity of $0$ as a fixed point of $g_{\lambda_0}$ with respect to the family $\mathcal{G}$ is $\mu_0^{fix}(g_{\lambda_0}$, $\mathcal{G})=\mu_0(f_{\lambda_0},\mathcal{F})$.
\end{definition}

In \cite{roussarie} Roussarie introduced the notion of cyclicity, measuring the number of limit cycles (isolated periodic orbits) that can be born from a certain set called limit periodic set by deformation of a given vector field in a family of vector fields. In the case of the family of Poincar\'e maps of the family of vector fields in a neighborhood of a limit periodic set, cyclicity is given by the above notion of multiplicity of a fixed point zero of $g_{\lambda_0}$ in the family of Poincar\'e maps. Due to the possible loss of differentiability of the family of Poincar\'e maps near a limit periodic set, the more general notion of multiplicity introduced in Definition \ref{mult} is necessary. 

\begin{remark}[Relation between classical multiplicity and multiplicity in a family for differentiable functions]\label{difcase}
 
 Note  that if $f$ is differentiable, $f\in \Diff^r[0,d)$,
 then the classical notion of multiplicity of $0$ as a zero of $f$, $\mu_0(f)\leq r$ as in \eqref{mul}, measures the maximal number of zeros of $f_\lambda$ that can appear near $0$, for $f_\lambda\in \mathcal{F}$ close to $f$, where $\mathcal{F}=\Diff^r[0,d)$ and the distance function is given by $d(f,g)=\sup_{k=0,\ldots,r}|f^{(k)}(0)-g^{(k)}(0)|$, i.e. $\mu_0(f)=\mu_0(f,\Diff^r[0,d))$. 
 Note that the number of zeros $\mu_0(f,\mathcal{F})$ that can appear by deformations $\mathcal{F}$ really depends on the family $\mathcal{F}$. Taking a family
 $\mathcal{F}$ of deformations bigger or smaller than $\Diff^r[0,d),$ it is easy to give examples with $\mu_0(f,\mathcal{F})$ bigger or smaller than $\mu_0(f)$.
 $($See e.g. Example 1.1.1 and Example 1.1.2 in \cite{mardesic}$)$.
 
 \begin{comment}
 that can appear near  
 
  the family of differentiable functions above notion of multiplicity in a family generalizes the classical notion of multiplicty \eqref{mul} in the case of 
 n \emph{the differentiable case}, $\mu_0(f)$ is defined at the beginning of  Section \ref{sec1}, then it is easy to see that 
  $\mu_0(f)=\mu_0(f, {\Diff}^r[0,d))$ for the differentiable Chebyshev scale $\mathcal{I}=\{1,x,x^2,\ldots\}$. In other words  $\mu_0(f)$ is equal to the multiplicity of $x=0$ of $f$ in a family   ${\Diff}^r[0,d)$ generated by $\mathcal{I}$, which is exactly the family of all   $C^r$ differentiable functions on $[0,d)$. 
\end{comment}

\end{remark}
We want to study non-differentiable functions having a special type of asymptotic behavior at $x=0$.
The definition of the following sequence of monomials and its properties is based on the notion of Chebyshev systems, see  \cite{J} and \cite{mardesic}, and the proofs therein.  A similar notion of asymptotic Chebyshev scale is mentioned in \cite{dumortier}.

\begin{definition}\label{cheb}
A finite or infinite sequence $\mathcal{I}=\{u_0,u_1,u_2,\ldots\}$ of functions of the class $C[0,d)\cap \Diff^r(0,d)$, $r\in\N\cup\{\infty\}$, is called \emph{a Chebyshev scale} if:
\begin{itemize} 
\item[i)] A system of differential operators $D_i$, $i=0,\ldots,r$, is well defined inductively by the following division and differentiation algorithm:
\begin{eqnarray*}\label{diff}
D_0(u_k)&=&\frac{u_k}{u_0},\\
D_{i+1}(u_k)&=&\frac{(D_i(u_k))'}{(D_i(u_{i+1}))'},\ i\in\mathbb{N}_0,
\end{eqnarray*}
for every $k\in\mathbb{N}_0$, except possibly in $x=0$ to which they are extended by continuity.
\item[ii)] The functions $D_i(u_{i+1})$ are strictly increasing on $[0,d)$, $i\in\mathbb{N}_0$.
\item[iii)] $\lim_{x\to 0} D_j u_i (x)=0$, for $j<i$, $i\in\mathbb{N}_0$.
\end{itemize} 

\noindent We call $D_i(f)$ \emph{the $i-$th generalized derivative of $f$ in the scale $\mathcal{I}$}.
\end{definition}

\medskip

\begin{definition} A function $f$ has a development in a Chebyshev scale $\mathcal{I}=\{u_0,\ldots,u_k\}$ of order $k$ if 
\begin{equation}\label{asymp}
f(x)=\sum_{i=0}^{k} \alpha_i u_i(x)+\psi(x),\quad \alpha_i\in\mathbb{R},
\end{equation}
and the generalized derivatives $D_i(\psi(x))$, $i=0,\ldots,k$, verify $D_i(\psi(0))=0$ (in the limit sense).
\end{definition}
Note that (in the limit sense) $D_i(f)(0)=\alpha_i$, $i=0,\ldots,k$.
\medskip

\begin{comment}
A family $\mathcal{L}$ of all functions having a development in a Chebyshev scale $\mathcal{I}=\{u_0,\ldots,u_k\}$ of order $k$ is called a family linearly generated by $\mathcal{I}$. The following proposition (see Marde\v si\' c \cite{mardesic}, Example 1.1.3) holds:
\begin{proposition}[Multiplicity in a linearly generated family]\ 
\begin{itemize}
\item[i)] Let $\I=\{u_0,\ldots,u_k\}$ and let $\mathcal{L}$ be the family linearly generated by $\I$. Let $f\in\mathcal{L}$. Then the multiplicity of $0$ as zero of $f$ in a family $\mathcal{L}$ is equal to $k_0\leq k$, $\mu_0(f,\mathcal{L})=k_0$, if and only if $D_i(f)(0)=0,\ i=0,\ldots,k_0-1, \text{ and } D_{k_0}(f)(0)\neq 0$.
\item[ii)] If $\mathcal{L}'\subset \mathcal{L}$, $\mathcal{L}$ as above, and $f\in\mathcal{L}'$, then $\mu_0(f,\mathcal{L}')\leq\mu_0(f,\mathcal{L})$ and it holds: if $D_i(f)(0)=0,\ i=0,\ldots,k_0-1, \text{ and } D_{k_0}(f)(0)\neq 0$, then $\mu_0(f,\mathcal{L}')\leq k_0$.
\item[iii)] If $\mathcal{L}'\subset \mathcal{L}$ and if $\mathcal{L}'$ satisfies ..., then the same conclusion holds as in the lineraly generated family: the multiplicity of $0$ as zero of $f\in\mathcal{L}'$ in a family $\mathcal{L}'$ is equal to $k_0\leq k$, $\mu_0(f,\mathcal{L}')=k_0$, if and only if $D_i(f)(0)=0,\ i=0,\ldots,k_0-1, \text{ and } D_{k_0}(f)(0)\neq 0$.
\end{itemize}
\end{proposition}
\end{comment}

Consider a family $\mathcal{F}=\{f_\lambda|\lambda\in \Lambda\}$ of functions having a uniform development of order $k$ in a family of Chebyshev scales
$\mathcal{I_\lambda}=(u_0(x,\lambda),\ldots,u_k(x,\lambda))$, i.e.
\begin{equation}\label{unasymp}
f_\lambda(x)=\sum_{i=0}^{k} \alpha_i(\lambda) u_i(x,\lambda)+\psi(x,\lambda),\quad \lambda\in W.
\end{equation}
The development is uniform in the sense that all generalized derivatives $D_j f_\lambda$, $j=0,\ldots,k$, can be extended by continuity to $x=0$ uniformly with respect to $\lambda$, and this extension is continuous as function of $\lambda$.

The following lemma generalizes Remark \ref{difcase}. It gives the connection between the index of the first nonzero coefficient in development of a function $f\in \mathcal{F}$ in a Chebyshev scale and the multiplicity of $0$ as a zero of $f$ in the family $\mathcal{F}$.
\begin{lemma}\label{multgen}
Let $\mathcal{I_\lambda}=(u_0(x,\lambda),\ldots,u_k(x,\lambda))$, $\lambda\in\Lambda$, be a family of Chebyshev scales and $\mathcal{F}=(f_ \lambda)$ a family of functions having a uniform development in the family of Chebyshev scales $\mathcal{I}_\lambda$ of order $k$ and $f_{\lambda_0}\in\mathcal{F}$. If the  generalized derivatives satisfy 
\begin{equation}\label{genderr}
D_i(f_{\lambda_0})(0)=0,\ i=0,\ldots,k_0-1, \text{ and } D_{k_0}(f_{\lambda_0})(0)\neq 0,\ k_0\leq k,
\end{equation} i.e., if $\alpha_{k_0}(\lambda_0)$ is the first nonzero coefficient in the development of $f_{\lambda_0}$,  then the multiplicity of $0$ as  zero of $f_{\lambda_0}$ in the family $\mathcal{F}$ is at most 
$k_0$ $($i.e.,
$\mu_0(f_{\lambda_0},\mathcal{F}) \leq k_0$$)$.

If moreover $\Lambda\subset \mathbb{R}^N$, $k_0\leq N$, and the matrix $(\frac{\partial \alpha_i}{\partial \lambda_j})_{i=0\ldots k_0-1,\ j=1\ldots k_0}(\lambda_0)$ is of maximal rank $($i.e., equal to $k_0)$, then  \eqref{genderr} is equivalent to\linebreak $\mu_0(f_{\lambda_0},\mathcal{F}) = k_0$. 
\end{lemma}

Proof of Lemma \ref{multgen} is based on Rolle's theorem
and the observation that dividing by nonzero functions the number of zeros is unchanged. 
If the matrix $(\frac{\partial \alpha_i}{\partial \lambda_j})$ is of maximal rank, then by the implicit function theorem one can consider parameters $\lambda_j$, $j=1,\ldots, k_0,$ as functions of $\alpha_0,\ldots,\alpha_{k_0-1},\lambda_{k_0+1},\ldots,\lambda_N$ and $\alpha_0,\ldots,\alpha_{k_0-1},\lambda_{k_0+1},\ldots,\lambda_N$ as new parameters. Then making a sequence of small deformations starting with $\alpha_{k_0-1},\ldots,\alpha_{k_0-2}$ etc, one can create $k_0$ small zeros in a neighborhood of $0$.
For the details of the proof, see Example 1.1.3 in \cite{mardesic}.

\begin{comment}
\begin{remark}[Multiplicity equal to $k_0$ for general enough families of functions]\label{general} 
Neki uvjeti na koef.... nesto sa Wronskijanom... Sto uniformna familija $\mathcal{F}=(f_\lambda)$ treba zadovoljavati da bi u njoj lezale i one funkcije na kojima se postize puni multiplicitet??? Vidjeti kako se konstruiraju...
\end{remark}
\end{comment}
\medskip

\begin{example}
$($Examples of Chebyshev scales on $[0,d)$$)$
\begin{enumerate}
\item[$i)$] differentiable case: e.g. $\mathcal{I}=\{1,x,x^2,x^3,x^4,\ldots\}$,
\item[$ii)$] non-differentiable case:
\begin{itemize}
\item[-]$\mathcal{I}=\{x^{\alpha_0},x^{\alpha_1},x^{\alpha_2},\ldots\}$, $\alpha_i\in\mathbb{R}$, $0<\alpha_0<\alpha_1<\alpha_2<\ldots$
\item[-]$\mathcal{I}=\{1,x(-\log x),x,x^2(-\log x),x^2,x^3(-\log x),x^3,\ldots\}$
\item[-] More generally, $\mathcal{I}$ can be any set of monomials of the type $x^k(-\log x)^l$, ordered by increasing flatness:
$$
x^i(-\log x)^j<x^k(-ln x)^l \text{ if and only if } (i<k) \text{ or } (i=j \text{ and } j>l).
$$
\end{itemize}
\item[$iii)$] For more general examples corresponding to Poincar\'e map at a homoclinic loop see Section \ref{homoclinic}.
\end{enumerate}
\end{example}

\medskip

\begin{definition}\label{comparable}
A function $f$ is \emph{weakly comparable to powers}, if there exist constants $m>0$ and $M>0$ such that
\begin{equation}\label{lowup}
m\leq x\cdot (\log f)'(x) \leq M,\ x\in(0,d).
\end{equation} 
We call the left-hand side of \eqref{lowup} the lower power condition and the right-hand side the upper power condition. 
A function $f$ is \emph{sublinear} if it satisfies lower power condition and $m>1$.
\begin{comment}
\end{comment}
\end{definition}

\begin{comment}
Our main results in Section \ref{sec2} will be applied to a special class of (possibly nondifferentiable in $x=0$) functions. Therefore we need to introduce the following definition.
\begin{definition}\label{admiss}
Let $f$ be a function of the class $C[0,d)\cap C^{\infty}(0,d)$ such that:
\begin{itemize}
\item[i)] $f$ is positive and strictly increasing on $(0,d)$, $\lim_{x\to 0}f(x)=0$,
\item[ii)] $f$ \emph{has an asymptotic expansion in some Chebyshev scale} \linebreak$\mathcal{I}=\{u_0,u_1,u_2,\ldots\}$, i.e., for every $n\in\mathbb{N}$, it holds
\begin{equation}\label{asy}
f(x)=\sum_{k=0}^n \alpha_k u_k + o(u_n);\ \ \alpha_k\in\mathbb{R},\ k=0,\ldots,n,
\end{equation}
where coefficients $\alpha_k$ do not depend on $n$,
\item[iii)] $f$ is \emph{weakly comparable to powers}, that is there exist constants $m>1$ and $M>0$ such that
\begin{equation}\label{lowup}
m\leq x\cdot (\log f)'(x) \leq M,\ x\in(0,d). 
\end{equation}
\end{itemize}
We then call the function $f$ \emph{admissible}.
\end{definition}
\end{comment}

A similar notion of comparability with power functions in Hardy fields appears in literature, see Fliess, Rudolph \cite{fliess} and Rosenlicht \cite{hardy}. A Hardy field $H$ is a field of real-valued functions of the real variable defined on $(0,d)$, $d>0$, closed under differentiation and with valuation $\nu$ defined in an ordered Abelian group. Let $f,\ g\in H$ be positive on $(0,d)$ and let $\lim_{x\to 0}f(x)=0$, $\lim_{x\to 0}g(x)=0$. If there exist integers $M, N\in\mathbb{N}$ and positive constants $\alpha,\beta>0$ such that 
\begin{equation}\label{class}
f(x)\leq \alpha g(x)^M \text{ and } g(x)\leq \beta f(x)^N,
\end{equation} 
it is said that $f$ and $g$ belong to the same comparability class/are comparable in $H$. \\ 
%\edz{Remark 2 utopljen u tekst, modificiran}
Let us state a sufficient condition for comparability from Rosenlicht \cite{hardy}, Proposition 4:
\begin{proposition}[Proposition 4 in \cite{hardy}]
Let $H$ be a Hardy field, $f(x),\ g(x)$ nonzero positive elements of $H$ such that  $\lim_{x\to 0}f(x)=0$, 
\linebreak $\lim_{x\to 0}g(x)=0$. If 
\begin{equation}\label{rosen}
\nu((\log f)')=\nu((\log g)'),
\end{equation} then $f$ and $g$ are comparable.
\end{proposition}

The condition \eqref{rosen} is equivalent to (see Theorem 0 in \cite{hardy})
\begin{equation}\label{strcomp}
\lim_{x\to 0}\frac{(\log f)'(x)}{(\log g)'(x)}=L,\ 0<L<\infty.
\end{equation}
Rosenlicht's condition \eqref{rosen}, i.e. \eqref{strcomp} is stronger than our condition \eqref{lowup} of weak comparability to powers. If $\lim_{x\to 0}\frac{(\log f)'(x)}{1/x}=L,\ 0<L<\infty$ (\eqref{lowup} obviously follows), then $f$ is comparable to power functions in the sense \eqref{class}.
\medskip

%\item[ii)]
Note that  condition $\eqref{lowup}$ excludes infinitely flat functions, but is nevertheless not equivalent to non-flatness. If $f$ is infinitely flat $($in the sense that all its derivatives tend to zero as $x\to 0$$)$, then it can easily be shown by L'Hospital rule that $\lim_{x\to 0}\frac{f(x)}{x^a}=0$, for every $a>0$ and, as a consequence, the inequality $(\ref{lowup})$ cannot be satisfied. The contrary is not true. There exist functions that are not infinitely flat, but nevertheless do not satisfy $(\ref{lowup})$. As an example, see the Example~\ref{constr} in Appendix. 
%\end{itemize}

\begin{example}[Weak comparability to powers and sublinearity]\ \\
$i)$ Functions of the form
$$
f(x)=x^\alpha (-\log x)^\beta,\ \alpha>0,\ \beta\in\mathbb{R},
$$
are weakly comparable to powers.\\This class obviously includes functions of the form $x^\alpha,\ x^\alpha(-\log x)^{\beta}$ and $\frac{x^\alpha}{(-\log x)^\beta}$, for $\alpha>0$ and $\beta>0$. If additionally $\alpha>1$, they are also sublinear.

\medskip
\noindent $ii)$ Functions of the form $$f(x)=\frac{1}{(-\log x)^\beta},\ \beta>0,$$
do not satisfy the lower power condition in $(\ref{lowup})$.

\medskip
\noindent $iii)$ Infinitely flat functions of the form $$f(x)= e^{-\frac{1}{x^\alpha}},\ \alpha>1,$$
do not satisfy the upper power condition in $(\ref{lowup})$, but they are sublinear.
\end{example}

\section{Main results}\label{sec2}
We now state the main theorems and their consequences.
\begin{theorem}\label{gensaus}
Let $f\in \Diff^r(0,d)$ be continuous on $[0,d)$, positive on $(0,d)$ and let $f(0)=f'(0)=0$.
Assume that $f$ is a sublinear function. 
%\edz{provjeriti dal nam stvarno netreba $m>1$.}
 Put $g=id-f$ and let $S^g(x_0)=\{x_n\ |\ n\in\mathbb{N}\}$ be an orbit of $g$, $x_0<d$.\\
The following formula for the length of the $\varepsilon$-neighborhood of the orbit $S^g(x_0)$ holds
\begin{equation}\label{gensausage}
|A_\varepsilon(S^g(x_0))| \simeq f^{-1}(\varepsilon).
\end{equation} 
%Let $f$ be differentiable on $(0,d)$, continuous at $0$ for some $d>0$, positive and strictly increasing on $[0,\delta)$ and $f(x)=o(x)$ as $x\to 0$. Let $f$ satisfy the condition $(\ref{lowup})$.
\end{theorem}
The sublinearity condition $m>1$ in the lower power condition cannot be omitted from Theorem~\ref{gensaus}. For counterexample, see Remark~\ref{superlin} in Appendix.

\bigskip

The following definition is a generalization of box dimension in non-differentiable case, according to a given Chebyshev scale. There  exists in the literature the notion of generalized Minkowski content, see He, Lapidus \cite{helapidus}, and \v Zubrini\' c, \v Zupanovi\' c  \cite{zuzu}. It is suitable in the situation  where the leading term of $|A_{\varepsilon}(U)|$ does not behave as a power function, and we introduce some functions usually called gauge functions. Driven by the result of Theorem~\ref{gensaus}, we follow the idea and define the generalized Minkowski content with respect to a family of gauge functions. By Theorem~\ref{gensaus}, the $\varepsilon$-neighborhood $|A_\varepsilon(S^g(x_0))|$ should be compared to the family obtained by inverting the given Chebyshev scale, $\{u_1^{-1}(\varepsilon),u_2^{-1}(\varepsilon),\ldots\}$. Comparing it to the powers of $\varepsilon$ as in the standard definition of the Minkowski content does not give precise enough information.
Next we define \emph{critical} Minkowski order which is close to the notion of box dimension. Its purpose is to contain information on the jump in the rate of growth of the length of $\varepsilon-$neighborhood of an orbit.

The upper (lower) generalized Minkowski content defined in Definition~\ref{generaldim} below can be viewed as function of $i$, $i=1,\ldots,\ell$. By Lemma~\ref{doubling}.$i)(b)$ in Section~\ref{sec3}, it is a discrete function which jumps only once from $+\infty$ to $0$ through some value $0\leq M\leq +\infty$. This is a behavior analogous to the behavior of the standard upper (lower) Minkowski content as function of $s$, see Section~\ref{sec1} and e.g. \cite{falconer}.
\begin{comment}
\begin{proposition}[Asymptotic behavior of inverses in a Chebyshev scale]\label{inverse}
Suppose $\I=\{u_0,\ldots,u_{\ell}\}$ is a Chebyshev scale, $\ell\in\mathbb{N}$, $u_i$ positive and strictly increasing, for $i\geq 1$. Let, for some $1\leq k_0\leq \ell$, the monomial $u_{k_0}$ satisfy the upper power condition. Then for the scale of inverses $\{u_{\ell}^{-1},u_{\ell-1}^{-1},\ldots,u_{k_0}^{-1},\ldots,u_1^{-1}\}$ it holds:
\begin{eqnarray*}
\lim_{y\to 0}\frac{u_i^{-1}(y)}{u_{k_0}^{-1}(y)}=0,&& 1\leq i<k_0,\\
\lim_{y\to 0}\frac{u_i^{-1}(y)}{u_{k_0}^{-1}(y)}=\infty,&& k_0<i\leq l.
\end{eqnarray*}
Specially, if all the monomials $u_j,\ j=1\ldots \ell$, in the scale $\I$ satisfy the upper power condition, then for every $1\leq i<j\leq\ell$ it holds:
$$
\lim_{y\to 0}\frac{u_i^{-1}(y)}{u_{j}^{-1}(y)}=0.
$$
\end{proposition}
\noindent For the proof of Proposition~\ref{inverse}, see Appendix.
\end{comment}

\medskip

Let $\mathcal{I}=\{u_0,u_1,\ldots\}$ be a Chebyshev scale such that monomials $u_i$ are positive and strictly increasing on $(0,d)$, for $i\geq 1$. 
Suppose that $f$ has a development in the Chebyshev scale of order $\ell$ on $[0,d)$ and moreover that $f$ satisfies assumptions from Theorem~\ref{gensaus} and upper power condition. Let $g=id-f$. We have the following definition:
\begin{definition}\label{generaldim}

\begin{comment}
Suppose that there exist $d>0$ and $k_0\in\mathbb{N}$ (and let $k_0$ be the smallest possible) such that for every $k\in\mathbb{N},\ k\geq k_0$, $u_k$ is positive and strictly increasing on $(0,d)$, $u_k(x)=o(x)$ and $u_k$ satisfies $(\ref{lowup})$. 
Let $f$ belong to a family $\mathcal{L}$ generated by $\mathcal{I}$,  and let $\lambda_0=\ldots=\lambda_{k_0-1}=0$ in  the expansion of $f$ and let its first nonzero coefficient be positive. Let $g=id-f$.\\
\end{comment}
By \emph{lower (upper) generalized  Minkowski content of $S^g(x_0)$ with respect to $u_i$},  $i=1,\ldots,\ell$, we mean
\begin{eqnarray*}
{\mathcal M_{*}}(S^g(x_0),u_i)&=&\liminf_{{\varepsilon}\to0}\frac{|A_{\varepsilon}(S^g(x_0))|}{u_i^{-1}({\varepsilon})},\\ {\mathcal M^{*}}(S^g(x_0),u_i)&=&\limsup_{{\varepsilon}\to0}\frac{|A_{\varepsilon}(S^g(x_0))|}{u_i^{-1}({\varepsilon})}
\end{eqnarray*}
respectively.
It can be easily seen that the behavior of the \linebreak $\varepsilon-$neighborhood and thus the definition is independent of the choice of the initial point $x_0<d$ in the basin of attraction of $0$. Therefore we define 
\begin{eqnarray*}
\underline{m}(g,\I)&=&\max\{i\ge 1\  |\ {\mathcal M_*}   (S^g(x_0),u_i) >0\},\\ 
\overline{m}(g,\I)&=&\max\{i\ge 1\  |\ {\mathcal M^*}   (S^g(x_0),u_i) >0\} 
\end{eqnarray*}
as the \emph{lower (upper) critical Minkowski order of $g$ with respect to the scale $\mathcal{I}$}, when a jump in lower (upper) generalized Minkowski content occurs.
If $\underline{m}(g,\I)=\overline{m}(g,\I)$, we call it \emph{critical Minkowski order with respect to the scale $\mathcal{I}$} and denote $m(g,\I)$.
\end{definition}

\begin{remark}[Box dimension and critical Minkowski order]\label{gbd}\ 

\begin{itemize}
\item[(i)] 
Using Lemma~\ref{doubling}, it can easily be seen that the upper and lower generalized Minkowski contents ${\mathcal M}(S^g(x_0),u_i)$ pass from the value $+\infty$, 
through a finite value and drop to $0$ as $i$ grows. Moreover, the critical index $i_0$ is the same for upper and lower content and therefore $m(g,\I)=i_0$.
\item[(ii)]
If $f\in Diff^r[0,d)$ is a differentiable function, then it has an asymptotic development in the differentiable Chebyshev scale, $\I=\{1,x,x^2,\ldots\,x^r\}$ of order $r$. The box dimension and critical Minkowski order are then directly related by the formula
\begin{equation}\label{Bm}
\dim_B(g)=1-\frac{1}{m(g,\I)}.
\end{equation}
Indeed, assume $f(x)\simeq x^k$, $1<k\leq r$. By  \eqref{sausage}, $|A_{\varepsilon}(S^g(x_0))|\simeq\varepsilon^{1/k}$. This gives $m(g,\I)=k$. On the other side, by definition, the box dimension is the value $s$ such that $1/k=1-s$.

\item[(iii)] By analogy with the differentiable case  \eqref{Bm}, we can define \emph{generalized box dimension of a function $g$  with respect to a Chebyshev scale $\mathcal{I}$  } by
$$
\dim_{GB} (g, \mathcal{I} )=1-\frac{1}{m(g,\I)}
$$
This definition is obviously independent of $x_0$ from the basin of attraction of $0$.
\end{itemize}
\end{remark}

The following theorem is a generalization of Theorem~\ref{neveda} to non-differentiable cases. Derivatives are replaced by generalized derivatives in a Chebyshev scale, and box dimension by similar notion of critical Minkowski order with respect to a Chebyshev scale. It shows that, in non-differentiable cases, the length of the $\varepsilon$-neighborhood should be compared with the inverted Chebyshev scale instead of the power scale to obtain multiplicity. 

Let $\mathcal{F}=\{f_\lambda|\lambda\in\Lambda\}$ be a family of functions on $[0,d)$ admitting a uniform asymptotic development \eqref{unasymp} in a family of Chebyshev scales $\I_\lambda=(u_0(x,\lambda),u_1(x,\lambda),\ldots)$, $\mathcal{G}=(g_\lambda)=\{id-f_\lambda| f_\lambda\in\mathcal{F}\}$. Let, {for} $\lambda=\lambda_0$, the monomials in the scale $\I=\I_{\lambda_0}$ be positive and strictly increasing on $(0,d)$, for $i\geq 1$.

\begin{theorem}\label{chebsaus}
 Let $f=f_{\lambda_0}$ be a function from the family $\mathcal{F}$ above, satisfying all assumptions of Theorem~\ref{gensaus} and the upper power condition. Let $g=id-f$. Then the following claims are equivalent:
\begin{enumerate}
\item[$(i)$] $D_i(f)(0)=0$, for $i=0,\ldots,k-1$, and $D_k(f)(0)>0$, for some $k\geq 1$ $($that is, $f\simeq u_k,$ for some $\ k\geq 1$$),$ 
\item[$(ii)$] $|A_\varepsilon(S^g(x_0)|\simeq u_k^{-1}(\varepsilon)$,
\item[$(iii)$] $m(g,\I)=k$.
\end{enumerate}
If moreover $\Lambda\subset\mathbb{R}^N$, $k\leq N$, and the matrix $(\frac{\partial \alpha_i}{\partial \lambda_j})_{i=0\ldots k-1,\ j=1 \ldots k}(\lambda_0)$ is of maximal rank $($i.e. equal to $k$$)$, then  $(1),\ (2)$ or $(3)$ is also equivalent to
\begin{enumerate}
\item[$(iv)$] $
\mu_0^{fix}(g,\mathcal{G})=k.$
\end{enumerate}
Without this regularity assumption, $(i)$, $(ii)$ or $(iii)$ implies
$$
\mu_0^{fix}(g,\mathcal{G})\leq k.
$$
\end{theorem}
On the importance of the upper power condition in Theorem~\ref{chebsaus}, see Remark~\ref{upperpower} in Appendix.

\bigskip

In the differentiable case, we notice that differentiation diminishes critical Minkowski order by 1. Let $f\in \Diff^r[0,d)$ be differentiable enough and suppose $\mu_0(f')>1$. Put $g=$id$-f$ and $h=$id$-f'$, then by Theorem~\ref{neveda} and Remark~\ref{gbd}.ii) we have $$m(h,\I)=m(g,\I)-1,$$ where $\mathcal{I}=\{1,x,x^2,\ldots,x^r\}$ is a differentiable Chebyshev scale.

The same property is valid in the non-differentiable case when $f$ has asymptotic development of some order in a Chebyshev scale, if the derivative is substituted by the generalized derivative in that scale. The following corollary is a direct consequence of Theorem~\ref{chebsaus}:
\begin{corollary}[behavior of the critical Minkowski order under differentiation]
Let $\mathcal{I}=\{u_0,u_1,\ldots,u_k\}$ be a Chebyshev scale and let $D_1(\I)$ denote the Chebyshev scale of first generalized derivatives of $\I$, that is, $D_1(\I)=\{D_1(u_1), D_1(u_2),\ldots,D_1(u_k)\}$. Suppose that $f$ has an asymptotic development in the scale $\I$ of order $k$ on $(0,d)$ and let $f$ and $D_1(f)$ satisfy assumptions from Theorem~\ref{gensaus} and the upper power condition. Let $g=$id$-f$, $h=$id$-D_1(f)$. It holds
$$
m(h,D_1(\I))=m(g,\I)-1.
$$ 
\end{corollary}

\bigskip 
%\edz{Da li je bolje ovako uklopljeno u tekst umjesto u Remarku, kako je bilo prije?}
Finally, let us explain the gain of introducing the notion of critical Minkowski order with respect to some Chebyshev scale over the standard box dimension. As we have mentioned before, the role of the box dimension and Minkowski content was to measure density of the orbit $S^g(x_0)$ around the fixed point $0$ by determining the rate of growth of $|A_{\varepsilon}(S^g(x_0))|$ as $\varepsilon\to 0$. If we consider, for example, functions $f_1(x)=x^k$ and $f_2(x)=x^k (-\log x)$, $k>1$, and compute standard box dimension of the orbits generated by $g_1=id-f_1$ and $g_2=id-f_2$, we get in both cases 
$$\dim_B(g_1)=\dim_B(g_2)=1-\frac{1}{k}.$$
Thus box dimension is equal for two functions obviously distinct in growth. This is unnatural, the difference can be noticed in the upper (lower) $(1-1/k)-$Minkowski content $\mathcal{M}^{1-1/k}$, which is zero for $g_2$ and greater than zero for $g_1$, thus signaling that the orbit generated by $g_1$ has bigger density around $0$ than the one generated by $g_2$. The reason lies in the fact that the box dimension and the Minkowski content are defined in a way that compares functions to power functions, but the functions $f_2(x)=x^k (-\log x)$  are \emph{not visible} in the scale of power functions. Precisely, by Theorem~\ref{gensaus}, $|A_{\varepsilon}(S^{g_2}(x_0))|\simeq f_2^{-1}(\varepsilon)$ does not behave as power of $\varepsilon$, but satisfies $\varepsilon^{1/k}<|A_{\varepsilon}(S^{g_2}(x_0))| <\varepsilon^{1/(k+\delta)}$, for every $\delta>0$, as $\varepsilon$ tends to 0. Therefore there is not much sense in comparing the lenght to powers of $\varepsilon$ as in the standard definition of box dimension. We should define some other Chebyshev scale in which $f_1$ and $f_2$ both have developments, for example $\mathcal{I}=\{1,x(-\log x),x,x^2(-\log x),x^2,\ldots\}$, and consider critical Minkowski orders with respect to this new scale instead of box dimensions. Then, we obtain distinct numbers: 
 $$m(g_1,\mathcal{I})=2k,\  m(g_2,\mathcal{I})=2k-1.$$
Therefore critical Minkowski order with respect to the appropriate scale is a more precise measure for density of the orbit in non-differentiable case then the box-dimension.

\begin{comment}

Let $\mathcal{I}=\{u_0,u_1,\ldots\}$ be a Chebyshev scale. Suppose that there exist $d>0$ and $k_0\in\mathbb{N}$  such that, for every $k\in\mathbb{N},\ k\geq k_0$, $u_k$ is an admissible function.
%positive and strictly increasing on $(0,d)$, $u_k(x)=o(x)$ and $u_k$ satisfies $(\ref{lowup})$. 
Let $f$ belong to a family $\mathcal{L}$ generated by $\mathcal{I}$ (i.e. admitting in $x=0$ asymptotic expansion in scale $\mathcal{I}$). Moreover, let $\lambda_0=\ldots=\lambda_{k_0-1}=0$, $\lambda_{k_0}>0$ in the expansion of $f$. Let $g=id-f$. Then it holds that the multiplicity
 $$\mu_{0}(f,\mathcal{L})=k \text{ if and only if }|A_\varepsilon(S^g(x_0))|\simeq u_k^{-1}(\varepsilon),\text{ as }\varepsilon\to 0,\ k\geq k_0.$$ 
 Moreover 
 \begin{equation}\label{boxdgen}
\mu_0 (f, \mathcal{L} )=\frac1{1-\dim_{GB}(g,\mathcal{I} )}.
\end{equation}
\end{comment}

\section{Proof of the main theorems}\label{sec3}
In the proof of Theorem~\ref{gensaus} and of Theorem~\ref{chebsaus} we need the following lemma:
\begin{lemma}[Inverse property]\label{doubling} 
Let $d>0$ and let $f, g\in C^1(0,d)$ be positive, strictly increasing functions on $(0,d)$.
\begin{enumerate}
\item[$i)$] If there exists a positive constant $M>0$ such that the upper power condition holds, \begin{equation}\label{up}x\cdot(\log f)'(x)\leq M,\ x\in(0,d),\end{equation} then
\begin{eqnarray}\label{sim}
&(a)&f^{-1}(y)\simeq g^{-1}(y),\text{ as }y\to 0 \text{ implies } f(x)\simeq g(x),\text{ as }x\to 0.\\
&(b)&\lim_{x\to 0}\frac{f(x)}{g(x)}=0\ (+\infty) \text{ implies } \lim_{y\to 0}\frac{f^{-1}(y)}{g^{-1}(y)}=+\infty\ (0).
\end{eqnarray}
\item[$ii)$] If there exists a positive constant $m>0$ such that the lower power condition holds, \begin{equation}\label{low}m\leq x\cdot(\log f)'(x),\ x\in(0,d),\end{equation} then
\begin{equation}\label{inverse}
f(x)\simeq g(x),\text{ as }x\to 0, \text{ implies } f^{-1}(y)\simeq g^{-1}(y),\text{ as }y\to 0.
\end{equation}
\end{enumerate}
\end{lemma}

\begin{proof}\ \

$i) a)$ From $f^{-1}\simeq g^{-1}$ we have that there exist constants $A<1$, $B>1$ and $\delta>0$ such that
$$
Ag^{-1}(y)\leq  f^{-1}(y)\leq Bg^{-1}(y) ,\ y\in(0,\delta).
$$
Putting $x=g^{-1}(y)$ and applying $f$ (strictly increasing) on the above inequality we get that there exists $\delta_1>0$ such that
\begin{equation}\label{eqq}
f(Ax)\leq g(x)\leq f(Bx), \ x\in(0,\delta_1).
\end{equation}
For each constant $C>1$ we have, for small enough $x$, 
\begin{eqnarray}\label{doup}
\log f(Cx)-\log f(x)&=&(\log f)'(\xi)(C-1)x\nonumber \\
&<&(\log f)'(\xi)(C-1)\xi,\quad \xi\in(x,Cx).
\end{eqnarray}
Combining $(\ref{up})$ and $(\ref{doup})$, we get that there exist constants $m_C>1$ and $d_C>0$ such that 
\begin{equation}\label{doubl}
\frac{f(Cx)}{f(x)}\leq m_C,\ x\in(0,d_C).
\end{equation}
Now using property $(\ref{doubl})$ and inequality $(\ref{eqq})$, for small enough $x$ we obtain
$$
\frac{1}{m_{1/A}}f(x)\leq g(x)\leq m_{B}f(x),
$$
i.e. $f(x)\simeq g(x)$, as $x\to 0$.

$i)\ b)$
Suppose $\lim_{x\to 0}\frac{f(x)}{g(x)}=+\infty$. We prove that $\lim_{y\to 0}\frac{f^{-1}(y)}{g^{-1}(y)}=0$ by proving that limit superior and limit inferior are equal to zero. Suppose the contrary, that is,
$$
\liminf_{y\to 0}\frac{f^{-1}(y)}{g^{-1}(y)}=M, \text{ for some $M>0$ or $M=\infty$}.
$$
By definition of limit inferior, there exists a sequence $y_n\to 0$, as $n\to\infty$, such that \begin{equation}\label{tend}\frac{f^{-1}(y_n)}{g^{-1}(y_n)}\to M, \text{ as }n\to\infty. \end{equation}
From \eqref{tend} it follows that there exist $n_0\in\mathbb{N}$ and $C>0$ such that
\begin{equation}\label{invineq}
g^{-1}(y_n)<C f^{-1}(y_n),\ n\geq n_0.
\end{equation}
Now, as in $i)(a)$, by a change of variables $x_n=g^{-1}(y_n)$, $x_n\to 0$, and applying $f$ (strictly increasing) on \eqref{invineq}, we get 
$$
m_C\ g(x_n)\geq f(x_n),\ n\geq n_0,\ x_n\to 0,\quad \text{ for }m_C>0,
$$
which is obviously a contradiction with $\lim_{x\to 0}\frac{f(x)}{g(x)}=+\infty$. Therefore 
$$
\liminf_{y\to 0}\frac{f^{-1}(y)}{g^{-1}(y)}=0.
$$
It can be proven in the same way that limit superior is equal to zero.

Now suppose $\lim_{x\to 0}\frac{f(x)}{g(x)}=0$. In the same way as above, we prove that $\lim_{y\to 0}\frac{g^{-1}(y)}{f^{-1}(y)}=0$.

\medskip

$ii)$ It is easy to see by change of variables $x=f^{-1}(y)$ that property $(\ref{low})$ of $f$ is equivalent to property $(\ref{up})$ of $f^{-1}$ and the statement follows from $i)$.
\end{proof}

\begin{remark}[Counterexamples in Lemma~\ref{doubling}]
To show that in Lemma \ref{doubling}. $i)$ the upper power condition $(\ref{up})$ is important, we can take, for example, functions $f(x)=e^{-\frac{1}{2x}}$ and $g(x)=e^{-\frac{1}{x}}$. They do not satisfy $(\ref{up})$ and, obviously,
$$
\lim_{x\to 0}\frac{f(x)}{g(x)}=\infty,\ \text{\ but\ } f^{-1}(y)=-\frac{1}{2\log y}\ \simeq\ g^{-1}(y)=-\frac{1}{\log y}.
$$
We can do the same for the lower power condition $(\ref{low})$ in Lemma \ref{doubling}. $ii)$ by considering, for example, $f(x)=-\frac{1}{\log x}$ and $g(x)=-\frac{1}{2\log x}$.
\end{remark}

\medskip

\noindent \emph{Proof of Theorem \ref{gensaus}}.\
\begin{comment}
\end{comment}

From lower power condition together with $f'(0)=0$, we get that $f(x)=o(x)$ and that $f(x)$ is strictly increasing on $(0,d)$. It can easily be checked that $x_n\to 0$ and $d(x_n,x_{n+1})\to 0$, as $n\to\infty$. Denote by $N_\varepsilon$ and $T_\varepsilon$ nucleus  and tail of the $\varepsilon-$neighborhood of the sequence, that are $\varepsilon-$neighborhoods of two subsets of the orbit satisfying the inequality $ d(x_n,x_{n+1})\leq 2\varepsilon $ for the nucleus, and  $d(x_n,x_{n+1})>2\varepsilon$ for the tail.
 Therefore,
\begin{equation}\label{area}
|A_\varepsilon(S^g(x_1))=|N_\varepsilon(S^g(x_1))|+|T_\varepsilon(S^g(x_1))|,
\end{equation}
where $|N_\varepsilon|$ is the length of the nucleus, and $|T_\varepsilon|$ the length of the tail of the $\varepsilon$-neighborhood. For more on notions of the tail and the nucleus of the $\varepsilon-$neighborhood of a set, see Tricot \cite{tricot}.

To compute the length, we have to find the index $n_\varepsilon\in\mathbb{N}$ such that
\begin{equation}\label{eps} 
f(x_{n_\varepsilon})< 2\varepsilon,\  f(x_{n_\varepsilon-1})\geq 2 \varepsilon,
\end{equation}
that is, the smallest index $n_\varepsilon$ such that $\varepsilon-$neighborhoods of the points $x_{n_\varepsilon},\ x_{n_\varepsilon+1},\ $etc start to overlap.

Then we have
\begin{eqnarray}
|N_\varepsilon|&=&x_{n_\varepsilon}+\varepsilon\label{nucleus},\\
|T_\varepsilon|&\simeq& n_\varepsilon\cdot\varepsilon.\label{tail}
\end{eqnarray}

First we estimate $|N_{\varepsilon}|$.
From $f(x)=o(x)$ we get
\begin{equation}\label{behinv}
\lim_{y\to 0}\frac{y}{f^{-1}(y)}=0.
\end{equation}
Since $f^{-1}$ is strictly increasing, from $(\ref{eps})$ we easily get $x_{n_\varepsilon}\simeq f^{-1}(2\varepsilon)$. Since $f$ satisfies the lower power condition, by Lemma \ref{doubling}.$ii)$ it follows $x_{n_{\varepsilon}}\simeq f^{-1}(\varepsilon)$. This, together with $(\ref{nucleus})$ and $(\ref{behinv})$, implies $|N_\varepsilon|\simeq f^{-1}(\varepsilon)$.

Now let us estimate the length of the tail, $|T_\varepsilon|$, by estimating $n_\varepsilon$.\\ 
Putting $\Delta x_n:=x_{n}-x_{n+1}$, from $x_{n+1}-x_n=-f(x_n)$ we get
\begin{equation}\label{nepsil}
\frac{\Delta x_n}{f(x_n)}=1 \text{ and }\sum_{n=n_0}^{n_\varepsilon}\frac{\Delta x_n}{f(x_n)}=\sum_{n=n_0}^{n_\varepsilon}1=n_{\varepsilon}-n_0\simeq n_\varepsilon, \text{ as }\varepsilon\to 0,
\end{equation}
for some fixed $n_0\in\mathbb{N}$. 

As in \eqref{relat} below, we get that $\frac{x_{n+1}}{x_n}$ tends to $1$, as $n$ tends to infinity, and thus we can choose the integer $n_0$ so that 
\begin{equation}\label{nzero}
Af(x_{n+1})<f(x_n)<Bf(x_{n+1}),\ \ n\geq n_0,
\end{equation}
for some constants $A,\ B>0$.

\begin{figure}
\centering
\hspace{-0.5cm}
\includegraphics[scale=0.6]{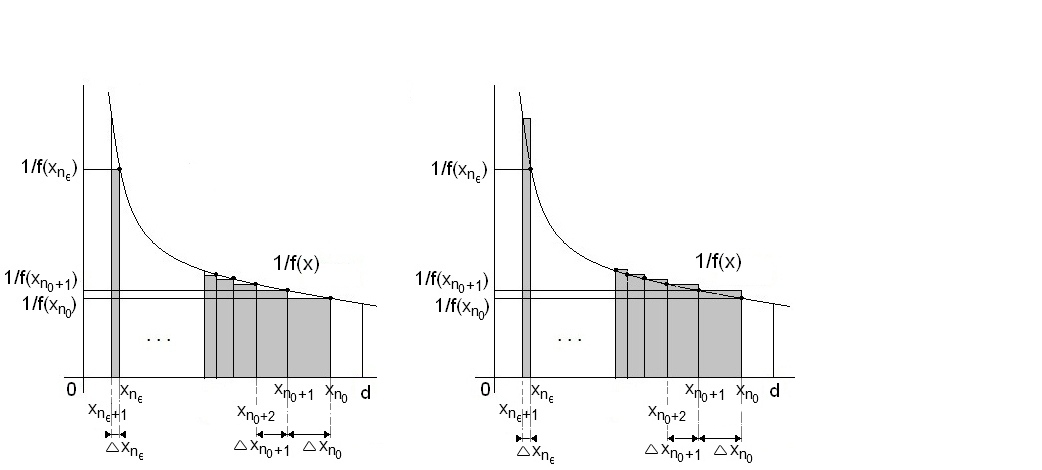}
\caption{\scriptsize Sums from \eqref{intap} as sums of areas of rectangles.}
\label{figproof}
\end{figure}

Since the function $\frac{1}{f(x)}$ is strictly decreasing on $(0,d)$ and $\lim_{x\to 0}\frac{1}{f(x)}=+\infty$, the sum $\sum_{n=n_0}^{n_\varepsilon}\frac{\Delta x_n}{f(x_n)}$ is equal to the sum of the areas of the rectangles in Figure~\ref{figproof}.$i)$ and, analogously, the sum $\sum_{n=n_0}^{n_\varepsilon}\frac{\Delta x_n}{f(x_{n+1})}$ is equal to the sum of the areas of the rectangles in Figure~\ref{figproof}.$ii)$. Therefore we have the following inequality:
\begin{equation}\label{intap}
\sum_{n=n_0}^{n_\varepsilon}\frac{\Delta x_n}{f(x_n)}\ \leq \int_{x_{n_{\varepsilon}+1}}^{x_{n_0}}\frac{dx}{f(x)} \leq \sum_{n=n_0}^{n_\varepsilon}\frac{\Delta x_n}{f(x_{n+1})}.
\end{equation}
From \eqref{nzero}, we get
\begin{equation}\label{upp}
\sum_{n=n_0}^{n_\varepsilon} \frac{\Delta x_n}{f(x_{n+1})}<B\sum_{n=n_0}^{n_\varepsilon}\frac{\Delta x_n}{f(x_n)},
\end{equation}
so finally, putting \eqref{upp} in \eqref{intap} and using \eqref{nepsil}, we get the following estimate for $n_\varepsilon$:
\begin{equation}\label{int}
n_\varepsilon\simeq \int_{x_{n_\varepsilon+1}}^{x_{n_0}}\frac{dx}{f(x)},\text { as }\varepsilon\to 0.
\end{equation}

Substituting $x=f^{-1}(y)$, from the lower power condition we get $$\frac{f^{-1}(y)}{y^2}\geq m\frac{(f^{-1})'(y)}{y}$$ and, consequently, for $y\in(0,f(d))$,
\begin{equation}\label{estim}
-\left(\frac{f^{-1}(y)}{y}\right)'=-\frac{(f^{-1})'(y)}{y}+\frac{f^{-1}(y)}{y^2}\geq (m-1)\cdot\frac{(f^{-1})'(y)}{y}.
\end{equation}

Now substitution $x=f^{-1}(s)$ in the integral $(\ref{int})$ together with $(\ref{estim})$ gives
\begin{equation}\label{neps}
n_\varepsilon\simeq \int_{f(x_{n_\varepsilon+1})}^{f(x_{n_0})}\frac{(f^{-1})'(s) ds}{s}\leq \frac{1}{m-1}\left(-\frac{f^{-1}(s)}{s}\right) \Big{|}_{f(x_{n_\varepsilon+1})}^{f(x_{n_0})}.
\end{equation}
It holds
\begin{eqnarray*}
\frac{f(x_{n_\varepsilon})}{f(x_{n_\varepsilon-1})}&=&\frac{f(x_{n_\varepsilon-1}-f(x_{n_\varepsilon-1}))}{f(x_{n_\varepsilon-1})}=\\
&=&\frac{f(x_{n_\varepsilon-1})+f'(\xi_\varepsilon)(-f(x_{n_\varepsilon-1}))}{f(x_{n_\varepsilon-1})}=1-f'(\xi_\varepsilon),
\end{eqnarray*}
for some $\xi_\varepsilon\in(x_{n_\varepsilon},x_{n_{\varepsilon}-1})$, so $f'(0)=0$ implies
\begin{equation}\label{relat}
\lim_{\varepsilon\to 0}\frac{f(x_{n_\varepsilon})}{f(x_{n_\varepsilon-1})}=1.
\end{equation} 
From $(\ref{eps})$ and $(\ref{relat})$, we now conclude that $f(x_{n_{\varepsilon}+1})\simeq \varepsilon$.
 Therefore $(\ref{neps})$ becomes
$$n_\varepsilon\leq C  \frac{f^{-1}(\varepsilon)}{\varepsilon},$$
for some $C>0$.
From \eqref{tail}, we have that $$|T_\varepsilon|\simeq n_\varepsilon\cdot\varepsilon\leq C_1\cdot f^{-1}(\varepsilon),$$ for some $C_1>0$ and $\varepsilon$ small enough. Together with $|N_\varepsilon|\simeq f^{-1}(\varepsilon)$ obtained above, this implies, using \eqref{area}, that
$$
|A_\varepsilon(S^g(x_1))|\simeq f^{-1}(\varepsilon), \text{ as }\varepsilon\to 0.
$$\qed

\bigskip

\noindent \emph{Proof of Theorem \ref{chebsaus}}.

We first prove that $(i)\Rightarrow (ii) \Rightarrow (iii)$. Suppose $D_i(f)(0)=0,\linebreak i=0,\ldots,k-1,\ D_k(f)(0)>0$, i.e., $f\simeq u_k$, as $x\to 0$. Theorem~\ref{gensaus} applied to $f$ gives $|A_\varepsilon(S^g(x_1))|\simeq f^{-1}(\varepsilon)$. Since $f\simeq u_k$, by Lemma~\ref{doubling}.$ii)$ we get $f^{-1}\simeq u_k^{-1}$ and therefore $|A_\varepsilon(S^g(x_1))|\simeq u_k^{-1}(\varepsilon)$. Since $u_k$ satisfies upper power condition, by Lemma~\ref{doubling} and Definition~\ref{generaldim} of critical Minkowski order (see Remark~\ref{gbd}.$ii)$), we get $m(g,\mathcal{I})=k$.

Now we prove that $(iii)\Rightarrow (ii) \Rightarrow (i)$. Suppose $m(g,\I)=k$ and $f\simeq u_l$, for some $l\neq k$. As above, we conclude $m(g,\mathcal{I})=l\neq k$, which is a contradiction. Therefore $f\simeq u_k$ and, again as above, $|A_\varepsilon(S^g(x_1))|\simeq u_k^{-1}(\varepsilon)$.  

By Lemma~\ref{multgen}, we conclude that $(i)$ implies $\mu_0^{fix}(g,\mathcal{G})\leq k$.
If moreover the condition on the maximal rank of the matrix $(\frac{\partial \alpha_i}{\partial \lambda_j})_{i=0 \ldots k-1,\ j=1\ldots k}(\lambda_0)$ is verified, by Lemma \ref{multgen} $(i)$ is equivalent to $(iv)$.

\hfill$\Box$

\section{Applications}\label{sec4}

\subsection{Cyclicity of limit periodic sets for planar systems}\

The number of limit cycles that bifurcate from a monodromic limit periodic set in an unfolding is equal to the multiplicity of the isolated fixed point $x=0$ of the Poincar\' e map in the family of Poincar\' e maps for the given unfolding, see e.g. Proposition 2 in \cite{dumortier}. For exact definitions of limit periodic set and cyclicity, see e.g. Roussarie \cite{roussarie}.
\begin{comment}
\end{comment}

Results from Section \ref{sec2} connect cyclicity of a limit periodic set in an unfolding and the rate of growth of the length of the $\varepsilon$-neighborhood of any orbit of the Poincar\' e map whose initial point is sufficiently close to the limit periodic set. This rate of growth is given by the critical Minkowski order with respect to the appropriate scale.

The orbit of the Poincar\' e map on a transversal to the limit periodic set is the intersection of the corresponding one-dimensional orbit of the vector field with the
transversal. Locally in a neighborhood of a point on the transversal, the structure of the one-dimensional orbit is that of the orbit of the Poincar\' e map by a segment.
Hence all interesting data of the one-dimensional orbit are given by the corresponding zero-dimensional orbit of the Poincar\' e map. 

%It was noted before in \cite{zuzu}, studying Hopf-Takens bifurcation, that box dimension of a spiral trajectory near a focus point or limit cycle gives information on cyclicity: the bigger the dimension, more periodic orbits are born in the perturbation.
Therefore, instead of considering the rate of growth of the area (2-Lebesgue measure) of the $\varepsilon$-neighborhood of an orbit of the field itself, as $\varepsilon\to 0$, which would be more natural in search of cyclicity of a limit periodic set, it is sufficient to consider the rate of growth of the length of the $\varepsilon$-neighborhood of an orbit of its Poincar\' e map.
\medskip

\begin{comment}Let us recall some definitions and terminology from Roussarie \cite{roussarie}:
\begin{definition}
Let $(X_\lambda)$ be a family of analytic vector fields on some compact surface $M$ (specially, unfolding of some analytic vector field $X_{\lambda_0}$). We say that a compact non-empty subset $\Gamma\subset M$ is a \emph{limit periodic set for the unfolding $(X_\lambda)$ of $X_{\lambda_0}$} if there exists a sequence of parameters $\lambda_n\to\lambda_0$, as $n\to\infty$, such that, for each $\lambda_n$, the vector field $X_{\lambda_n}$ has limit cycle $\gamma_{\lambda_n}$ with the property
$$
d(\gamma_{\lambda_n},\Gamma)\to 0, \text{ as $n\to\infty$}
$$
in the appropriate topology $($for details, see \cite{roussarie},  p. 18-19$)$.
\end{definition}

The \emph{number of limit cycles which bifurcate from limit periodic set $\Gamma$ of $X_{\lambda_0}$ in the analytic unfolding $(X_\lambda)$} is given by the following definition:
\begin{definition}
Let $\Gamma$ be a limit periodic set of $(X_\lambda)$ for value $\lambda_0$. For each $\varepsilon>0, \delta>0$ we define
$$
N(\delta,\varepsilon)=\sup \{\mbox{number of limit cycles } \gamma \mbox{ of } X_\lambda: d(\gamma,\Gamma)<\varepsilon,\ d(\lambda,\lambda_0)<\delta \} .
$$
Then the cyclicity of $\Gamma$ in the unfolding $(X_\lambda)$ is given by
$$
Cycl (X_\lambda,\Gamma)=\inf_{\varepsilon,\delta} N(\delta,\varepsilon).
$$
\end{definition}
\end{comment}

Let $\Gamma$ be the stable limit periodic set for the analytic unfolding $X_\lambda$. We consider only the unfoldings of finite codimension such that the family of Poincar\' e maps $g_\lambda(x),\ x\in[0,d),$ for the unfolding is well defined and different from identity on the transversal to vector field $X_\lambda$ in neighborhood of $\Gamma$. 

Let us recall that the function $f_\lambda$=id$-g_\lambda$ is called the displacement function. The main idea is to find the family of Chebyshev scales $\mathcal{I_\lambda}$ such that the family $f_\lambda$ has a uniform development of some order in the family $\I_\lambda$, as was introduced in Section~\ref{sec2}. Suppose $\Gamma$ is a limit periodic set of $X_{\lambda_0}$, for the parameter value $\lambda=\lambda_0$. Then, by Theorem~\ref{chebsaus}, the critical Minkowski order of $g_{\lambda_0}$ with respect to the scale $\I_{\lambda_0}$, $m(g_{\lambda_0}, I_{\lambda_0})$, is an upper bound on cyclicity of $\Gamma$ in the unfolding $X_\lambda$. 

In the sequel, limit periodic sets are stable limit cycles, nondegenerate stable focus points and stable homoclinic loops. The family of displacement functions $f_\lambda$ for the unfolding has a uniform asymptotic development in a family of appropriate Chebyshev scales and is analytic in $x=0$ in first two cases (differentiable cases) and non-differentiable in $x=0$ in the case of homoclinic loop. 

\subsubsection{Differentiable case, limit cycle}\

Suppose that $X_{\lambda_0}$ has a stable or semistable limit cycle $\Gamma$ and let $X_\lambda$ be an arbitrary analytic unfolding of $X_{\lambda_0}$.\\
There exists neighborhood $W$ of $\lambda_0$ such that the displacement function $f_\lambda$ is analytic on $[0,d)$, for $\lambda\in W$, and $f_{\lambda_0}(0)=0$. Expanding $f_\lambda(x)$ in Taylor series, we get
\begin{equation}\label{tayllimit}
f_\lambda(x)=\alpha_0(\lambda)+\alpha_1(\lambda)x+\alpha_2(\lambda)x^2+\alpha_3(\lambda)x^3+\ldots, \ \lambda\in W.
\end{equation} 

The family $f_\lambda$ has a uniform asymptotic development in the Chebyshev scale $$\I=\{1,x,x^2,\ldots,x^\ell\},$$ of any order $\ell\in\mathbb{N}$. %Let $\mathcal{G}=\{g_\lambda=id-f_\lambda\}$ be the family of Poincar\' e maps. By Theorem~\ref{chebsaus}, if $f_{\lambda_0}\simeq x^{k}$, for some %$2\leq k\leq\ell$, then the critical Minkowski order of $g_{\lambda_0}$ with respect to $\I_{\lambda_0}$ is equal to $k$, %$m(g_{\lambda_0},\I_{\lambda_0})=k$. The cyclicity of the limit cycle in the unfolding $X_\lambda$ is equal to %$\mu_0^{fix}(g_{\lambda_0},\mathcal{G})\leq k$. If moreover the unfolding $(X_\lambda)$ is general enough so that the regularity condition from %Theorem~\ref{chebsaus} is satisfied, then  the cyclicity $\mu_0^{fix}(g_{\lambda_0},\mathcal{G})=k$.

By Theorem~\ref{gensaus}, the length of the $\varepsilon$-neighborhood of an orbit of the Poincar\' e map around the limit cycle should be compared to the inverted scale, $\{\varepsilon,\varepsilon^{1/2},\varepsilon^{1/3},\ldots\}$, to obtain an upper bound on the cyclicity. Let $\mathcal{G}=\{g_\lambda=id-f_\lambda\}$ be the family of Poincar\' e maps. By Theorem~\ref{chebsaus}, if $f_{\lambda_0}\simeq x^{k}$, for some $2\leq k\leq\ell$, then $|A_\varepsilon|\simeq \varepsilon^{1/k}$ and the critical Minkowski order of $g_{\lambda_0}$ with respect to $\I_{\lambda_0}$ is equal to $k$, $m(g_{\lambda_0},\I_{\lambda_0})=k$. The cyclicity of the limit cycle in the unfolding $X_\lambda$ is equal to $\mu_0^{fix}(g_{\lambda_0},\mathcal{G})\leq k$. If moreover the unfolding $(X_\lambda)$ is general enough so that the regularity condition from Theorem~\ref{chebsaus} is satisfied, then  the cyclicity $\mu_0^{fix}(g_{\lambda_0},\mathcal{G})=k$.

\subsubsection{Differentiable case, weak focus}\

Suppose $x_0$ is a stable weak focus of $X_{\lambda_0}$(that is, $DX_{\lambda_0}(x_0)$ has two conjugate complex eigenvalues without the real part). Suppose $X_{\lambda}$ is an arbitrary analytic unfolding of $X_{\lambda_0}$.

There exists neighborhood $W$ of $\lambda_0$ such that, for $\lambda \in W$, the displacement function $f_\lambda(x)$ is analytic in $0$ and $f_\lambda(0)=0$. Therefore we can expand $f_\lambda$ in Taylor series around $0$ and, by symmetry argument around focus point, we get that the leading monomials can only be the ones with odd exponents:
\begin{equation}\label{focus}
f_{\lambda}(x)=\beta_1(\lambda)(x+g_1(\lambda,x))+\beta_3(\lambda)(x^3+g_3(\lambda,x))+\beta_5(\lambda)(x^5+g_5(\lambda,x))+\ldots,
\end{equation}
where $g_i(\lambda,x)$ denotes some linear combination of monomials from Taylor expansion of order strictly greater than $x^i$ and with coefficients depending on $\lambda$.

The family of displacement functions $f_\lambda$ has a uniform asymptotic development in a family of Chebyshev scales $\I_{\lambda}$ of some order:
$$
\mathcal{I}_\lambda=\{x+g_1(\lambda,x),\ x^3+g_3(\lambda,x),\ x^5+g_5(\lambda,x),\ldots,\ x^{2\ell+1}+g_{2\ell+1}(\lambda,x)\}.
$$

To obtain an upper bound on the cyclicity of the focus, by Theorem~\ref{chebsaus}, the length of the $\varepsilon$-neighborhood of the discrete orbit of the Poincar\' e map around the origin should be compared to the inverted scale $\{\varepsilon,\ \varepsilon^{1/3},\ \varepsilon^{1/5},\ldots\}$ of $\mathcal{I}_{\lambda_0}$. We proceed as in the example above.
%For the value $\lambda_0$ it holds 
%$\beta_1(\lambda_0)=0$ and therefore $f_{\lambda_0}(x)\simeq x^{2k+1}, \text{ as } \linebreak  x\to 0$, for some $k\geq 1$. 

%Let $\mathcal{G}=\{g_\lambda=\text{id}-f_\lambda\}$ be a family of Poincar\' e maps. 
%By Theorem~\ref{chebsaus}, the critical Minkowski order of $g_{\lambda_0}$ with respect to $\mathcal{I}_{\lambda_0}$ is equal to $k$, %$m(g_{\lambda_0},\I_{\lambda_0})=k$, and cyclicity of the focus is equal to $\mu_0^{fix}(g_{\lambda_0},\mathcal{G})\leq k$. If moreover the unfolding %$(X_\lambda)$ is general enough so that the regularity condition from Theorem~\ref{chebsaus} is satisfied, then the cyclicity %$\mu_0^{fix}(g_{\lambda_0},\mathcal{G})=k$.

%It is not difficult to check that $\mathcal{I}(\lambda)$, $\lambda\in W$, and $f_{\lambda_0}$  meet the conditions of Corollary \ref{chebsaus} and %thus knowing $|A_{\varepsilon}(S^{f_{\lambda_0}}(x_1))|\simeq u_k(\lambda_0)^{-1}$, we can conclude that cyclicity of such focus is less than or equal %to $k$, $k\geq 1$.

\subsubsection{Non-differentiable case, homoclinic loop}\label{homoclinic}\

Suppose $X_{\lambda_0}$ has a stable homoclinic loop with the hyperbolic saddle point at the origin, with $0x$ as unstable and $0y$ as stable manifold, and such that the ratio of hyperbolicity of the saddle is $r(\lambda_0)=1$ (i.e. $DX_{\lambda_0}(0)$ has eigenvalues of the same absolute value, but of different sign). Suppose $X_\lambda$ is an analytic unfolding of $X_{\lambda_0}$ and that, for $\lambda\in W$, each $X_\lambda$ has a hyperbolic saddle of ratio $r(\lambda)$ at the origin, with the same stable and unstable manifolds.

We consider the family of Poincar\' e maps $\mathcal{G}=g_\lambda$ and the family of displacement functions $f_\lambda=\text{id}-g_\lambda$, $x\in(0,\delta)$, on a transversal to stable manifold near the origin, as in Chapter 5 in Roussarie \cite{roussarie}. The family cannot be extended analytically to $x=0$ due to nondifferentiability in $x=0$ and the following asymptotic expansion in $x=0$ holds instead (see \cite{roussarie}, Section 5.2.2):
\begin{eqnarray}\label{exploop}
f_\lambda(x)&=&\beta_0(\lambda)+\alpha_1(\lambda)[x\omega(x,\alpha_1(\lambda))+g_1(x,\lambda)]+ \nonumber\\
&+&\beta_1(\lambda)x+\alpha_2(\lambda)[x^2\omega(x,\alpha_1(\lambda))+g_2(x,\lambda)]+\beta_2(\lambda)x^2+\ldots+\nonumber \\
&+&\beta_n(\lambda)x+\alpha_n(\lambda)[x^n\omega(x,\alpha_1(\lambda))+g_n(x,\lambda)]+\beta_n(\lambda)x^n+ o(x^n),\nonumber \\
&&\hspace{10cm} n\in\mathbb{N},
\end{eqnarray}
where $\alpha_1(\lambda)=1-r(\lambda)$, $g_i(x,\lambda)$ denotes linear combination in the monomials of the type $x^k\omega^l$ of strictly greater order than $x^i\omega$ (order on monomials is defined by increasing flatness, $x^i \omega^j<x^k \omega^l$ if ($i<k$) or ($i=k$ and $j>l$)) and
$$
\omega (x,\alpha)=\left\{
\begin{array}{ll} \frac{x^{-\alpha}-1}{\alpha} &\text{ if } \alpha\neq 0,\\ -\log x&\text { if }\alpha=0.
\end{array}
\right.$$

The family of displacement functions $f_\lambda$ has obviously an uniform asymptotic development in the following family of Chebyshev scales $\I_\lambda$ of some order:
$$
\mathcal{I}_\lambda=\{1,x\omega(x,\alpha_1(\lambda))+g_1(x,\lambda),x,x^2\omega(x,\alpha_1(\lambda))+g_2(x,\lambda),x^2,\ldots\}.
$$
If we take $\lambda=\lambda_0$ in $(\ref{exploop})$, we get the following expansion for $X_{\lambda_0}$ ($\alpha_1(\lambda_0)=0,\ f_{\lambda_0}(0)=0$):
\begin{eqnarray}\label{pt}
f_{\lambda_0}(x)&=&\beta_1(\lambda_0) x+\alpha_2(\lambda_0)x^2\omega(x,0)+\alpha_3(\lambda_0)x^3\omega(x,0)+...=\nonumber \\
&=&\beta_1(\lambda_0) x+\alpha_2(\lambda_0)x^2(-\log x)+\alpha_3(\lambda_0)x^3(-\log x)+...
\end{eqnarray}

The length of the $\varepsilon$-neighborhood should be compared to the inverted scale of $\mathcal{I}_{\lambda_0}$ to obtain information on cyclicity. The critical Minkowski order signals the moment the comparability occurs.
By Theorem~\ref{chebsaus}, if $f_{\lambda_0}(x)\simeq x^k$, as $x\to 0$, $k\geq 2$, then the critical Minkowski order is equal to $2k$, $m(g_{\lambda_0},\I_{\lambda_0})= 2k$; if $f_{\lambda_0}\simeq x^k(-\log x)$, $k\geq 2$, then the critical order is equal to $2k-1$, $m(g_{\lambda_0},\I_{\lambda_0})= 2k-1$. Consequently, the cyclicity of the loop is less than or equal to $2k$, $2k-1$ respectively. Equality can be obtained if the unfolding $(X_\lambda)$ is general enough so that the regularity condition from Theorem~\ref{chebsaus} is satisfied.

\subsubsection{Non-differentiable case, Hamiltonian hyperbolic 2-cycle with constant hyperbolicity ratios.}\label{poli}\

Suppose $(X_\lambda)$ is an unfolding of a Hamiltonian hyperbolic 2-cycle of the field $X_{\lambda_0}$ in which at least one separatrix remains unbroken. Such a situation appears for polycycles having part of the line at infinity as the unbroken separatrix. Suppose that the ratios of hyperbolicity of both saddles $S_1$ and $S_2$ at $\lambda=\lambda_0$ are $r_1(\lambda_0)=r_2(\lambda_0)=1$. The breaking parameter of the breaking separatrix is denoted by $\beta_1(\lambda)$ $(\beta_1(\lambda_0)=0)$. By $x\in(0,\delta)$ we parametrize the (inner side) of the transversal to the stable manifold of one of the saddles, and we choose the saddle whose stable manifold is on the unbroken separatrix, say $S_1$. In search of cyclicity, instead of considering fixed points of Poincar\' e maps $f_\lambda$ on $(0,\delta)$, for simplicity we can consider zero points on $(0,\delta)$ of the family of maps  
$$
\Delta_\lambda(x)=D_2^\lambda\circ R_2^\lambda(x)-R_1^\lambda\circ D_1^\lambda(x),
$$
where $D_1$ and $D_2$ represent Dulac maps of the saddle $S_1$, $R_1$ is the regular map along the broken separatrix and $R_2$ the regular map along the unbroken separatrix. Obviously, $R_1^\lambda(0)$ equals the breaking parameter of the separatrix, $\beta_1(\lambda)$, and $R_2^\lambda(0)=0$, on the unbroken separatrix. Using the developments of Dulac maps from \cite{roussarie}, just like in the above example of the saddle loop, $\Delta_\lambda$ has the uniform development in the monomials from the two Chebyshev scales, $\mathcal{I}_\lambda^1$ and $\mathcal{I}_\lambda^2$ below, since the developments for $D_2^\lambda\circ R_2^\lambda(x)$ and $R_1^\lambda\circ D_2^\lambda(x)$ are subtracted:
\begin{align*}
\mathcal{I}_{\lambda}^1&=\{1,x\omega_1(x,\alpha_1(\lambda)),x,x^2\omega_1^2(x,\alpha_1(\lambda)),x^2\omega_1(x,\alpha_1(\lambda))),x^2,\\
&\qquad x^3\omega_1^3(x,\alpha_1(\lambda)),x^3\omega_1^2(x,\alpha_1(\lambda)),x^3\omega_1(x,\alpha_1(\lambda)),x^3,\ldots\},\\
\mathcal{I}_{\lambda}^2&=\{1,x\omega_2(x,\alpha_2(\lambda)),x,x^2\omega_2^2(x,\alpha_2(\lambda)),x^2\omega_2(x,\alpha_2(\lambda))),x^2,\\&\qquad x^3\omega_2^3(x,\alpha_2(\lambda)),x^3\omega_2^2(x,\alpha_2(\lambda)),x^3\omega_2(x,\alpha_2(\lambda)),x^3,\ldots\}.
\end{align*}
For the development, see e.g. \cite{caubergh}.
For each monomial $x^k\omega_{i}^l$, $k\geq 1,\ l\geq 0$, it necessarily holds that $k\geq l$, $\alpha_1(\lambda)=1-r_1(\lambda)$, $\alpha_2(\lambda)=1-r_2(\lambda)$, and  $\omega_1$ and $\omega_2$ are as defined in the section above. They are known as independent compensators, since they are not comparable by flatness, and thus disable the concatenation of $\mathcal{I}_{\lambda}^1$ and $\mathcal{I}_{\lambda}^2$ in one Chebyshev scale. 

If we additionally suppose that the ratios of hyperbolicity $r_1(\lambda_0)$ and $r_2(\lambda_0)$ are preserved throughout the unfolding, then we have $$\omega_1(x,\alpha_1(\lambda))=\omega_2(x,\alpha_2(\lambda))=-\log x,\text{ for all $\lambda$}.$$
In this case the Chebyshev scale in which all of $\Delta_\lambda$ from the unfolding $(X_\lambda)$ have the uniform development is
\begin{align*}
\mathcal{I}=&\{1,x,x^2(-\log x)^2,x^2(-\log x),x^2,\\
&\qquad x^3(-\log x)^3,x^3(-\log x)^2,x^3(-\log x),x^3,\ldots\}.
\end{align*}
To see the number of limit cycles that can arise in the unfolding of the hyperbolic 2-cycle in $X_{\lambda_0}$, by Theorem~\ref{chebsaus}, the length of $\varepsilon$-neighborhood of the discrete orbit of $\Delta_{\lambda_0}(x)$ should be computed numerically and compared to the inverted scale of $\mathcal{I}$. The index $i$ for which $|A_\varepsilon|\simeq u_i^{-1}(\varepsilon)$ holds, in the article called critical Minkowski order $m(\Delta_{\lambda_0},\mathcal{I})$, represents an upper bound on the number of limit cycles that can appear in the unfolding $(X_\lambda)$ of $X_{\lambda_0}$. 

Let us note here that this upper bound is not necessarily optimal, since the scale $\mathcal{I}$ is taken to be the largest possible for a given problem. Better results on upper bound are obtained in \cite{dumortier}, using asymptotic developments of Abelian integrals, and in \cite{gavrilov}. In \cite{gavrilov}, the upper bound is given in terms of characteristic numbers of holonomy maps, not using asymptotic development of the Poincar\' e map.
\subsection{Abelian integrals}\ 
%\edz{gledamo samo u $R^2$, ne u $R^n$}

\emph{Abelian integrals} on $1$-cycles are integrals of polynomial $1-$form $\omega$ along the continuous family of cycles of the polynomial Hamiltonian field, lying in the level sets of the Hamiltonian $H$, $\delta_t\subset \{H=t\}$,
$$
I_{\omega}(t)=\int_{\delta_t}\omega.
$$

Suppose that the value $t=0$ is a critical value for the Hamiltonian field in $\R^2$, such that there exists $d>0$ and a continuous family of cycles belonging to the level sets $\{H=t\}$,\ $t\in (0,d)$. Then we have the following asymptotic expansion at $t=0$ (see Arnold \cite{arnold}, Ch. 10, Theorem 3.12 and Zoladek \cite{zoladek}, Ch. 5):
\begin{equation}\label{e}
I_{\omega}(t)=\sum_{\alpha}\sum_{k=0}^{1}a_{k,\alpha}(\omega) t^{\alpha}(-\log t)^k,
\end{equation}
where $\alpha$ runs over an increasing sequence of nonnegative rational numbers depending only on Hamiltonian $H(x,y)$ (such that $e^{2\pi i \alpha}$ are eigenvalues of monodromy operator of the singular value) and $a_{k,\alpha}\in\mathbb{R}$.
% !!!$k$ IDE DO n-1 JER JE n MAKSIMALNI RED CELIJE KOJA ODGOVARA SV. VRIJEDNOSTI OPERATORA MONODROMIJE $e^{2\pi i \alpha}$ ZA POJEDINI $\alpha$!!!\\ 
%OVI RAZVOJI, VIDJETI Arnold p.275, VRIJEDE ZA BAZU \newline 1-HOMOLOGIJE $t$-NIVO SKUPA, $H^1(H^{-1}(t))$, DA LI NASA FAMILIJA CIKLUSA KOJI LEZE U $t$-NIVO SKUPOVIMA $(\delta_t\subset \{H=t\})$ PRIPADA TIM 1-HOMOLOGIJAMA???
\medskip

Obviously, the corresponding Chebyshev scale for this problem is:
\begin{equation}
\mathcal{I}=\{t^{\alpha_1}(-\log t),t^{\alpha_1},
t^{\alpha_2}(-\log t),t^{\alpha_2},\ldots,
t^{\alpha_m}(-\log t),t^{\alpha_m},\ldots\}.
\end{equation}

It makes sense to compute critical Minkowski order of the orbit $S^g(x_0)$, comparing the length of $\varepsilon$-neighborhood of $g(t)=t-I_{\omega_0}(t)$ with the inverted scale of $\mathcal{I}$, to obtain the multiplicity of an Abelian integral in a family of integrals.\\

In $\mathbb{R}^2$, Abelian integrals have been used as a tool for determining cyclicity of vector fields, considering them as perturbation of Hamiltonian field (for details and examples see e.g. Zoladek \cite{zoladek}, Ch. 6). 

Suppose we have the following $\eta-$perturbed Hamiltonian system,
\begin{eqnarray}\label{perturb}
\dot{x}&=&\frac{\partial H}{\partial y}+\eta P(x,y,\eta),\nonumber\\
\dot{y}&=-&\frac{\partial H}{\partial x}+\eta Q(x,y,\eta),
\end{eqnarray} 
where $P,\ Q,\ H$ are polynomials and $\eta>0$.

\noindent Let $\omega=Q dx- P dy$ be the polynomial $1-$form defined by $P,\ Q$.

Let $t=0$ be a critical value of Hamiltonian, and let $S$ be a transversal to the family of cycles $(\delta_t\subset\{H=t\})$ on small neighboourhood of $t=0$, parametrized by $t\in[0,d)$. Then (see e.g. Zoladek \cite{zoladek}, Ch. 6) the displacement function on $S$ of the perturbed Hamiltonian field is given by
\begin{equation}\label{app}
f_{\eta}(t)=\eta I_\omega(t)+o(\eta),
\end{equation}
i.e., Abelian integral is the first approximation of the displacement function, for $\eta$ small enough. Here we suppose that $I_\omega(t)$ is not identically equal to zero, i.e. that $\omega$ is not relatively exact.

On some segment $[\alpha,\beta]\subset (0,d)$ away from critical value $t=0$, it is known that the number of zeros of Abelian integral gives an upper bound on the number of zeros of the displacement function $f_{\eta}(t)$ on $[\alpha,\beta]$ of the perturbed system~$(\ref{perturb})$, for $\eta$ small enough (both counted with multiplicities), i.e. on the number of limit cycles born in perturbed system $(\ref{perturb})$ in the area $\bigcup_{t\in[\alpha,\beta]}\delta_t$, for $\eta<\eta_0$ small enough (for this result, see e.g. \cite{spain}, Theorem 2.1.4).

However, the problem arises if we approach the critical value $t=0$ and the result cannot be applied to the whole interval $[0,d)$. In some systems, some limit cycles visible as zeros of displacement function are not visible as zeros of corresponding Abelian integral, because sometimes the approximation $(\ref{app})$ is not good enough. One of the examples is the perturbation of the Hamiltonian field in the neighborhood of the saddle polycycle with 2 or more vertices, see Dumortier, Roussarie \cite{dumortier}. Abelian integrals near hyperbolic  polycycles have an expansion linear in $\log t$, see expansion $(\ref{e})$ or \cite{dumortier}, Proposition 1. On the other hand, see Roussarie \cite{roussarie}, the asymptotic expansion of the displacement function near the saddle polycyle with more than one vertex involves also powers of $\log t$ greater than $1$. 

In the neighborhood of the center singular point and of the saddle loop ($1-$saddle polycycle) of the Hamiltonian field, however, the multiplicity of corresponding Abelian integral gives correct information about cyclicity, see e.g. Dumortier, Roussarie \cite{dumortier}, Theorem 4.  

\bigskip

\subsection{One example out of scope of Theorem~\ref{chebsaus}}\

%\edz{Ovdje je komentar o plosnatim funkcijama pomaka kod akumulacije ciklusa na koje je teorem neprimijenjiv.}
At the very end, let us note that in the former examples we have used critical Minkowski order which reveals the rate of growth of $\varepsilon$-neighborhood of the orbit generated by Poincar\' e map around limit periodic set to conclude about the cyclicity of the set. The connection is given by Theorem~\ref{chebsaus} through the notion of multiplicity of the fixed point zero which is equal to cyclicity. 

From the assumptions of Theorem~\ref{chebsaus} it is visible that the theorem cannot be applied to displacement functions which are infinitely flat and therefore not comparable to powers (see the paragraph after Definition~\ref{comparable}). We have noticed that this restriction of Theorem~\ref{chebsaus} to functions that are not infinitely flat makes sense in applications. As an example, we can take the accumulation of limit cycles on the saddle-node polycycle, a case which obviously should not meet the conditions of Theorem~\ref{chebsaus} because multiplicity and cyclicity should not turn out finite. It is interesting that in this case the displacement function is infinitely flat, and therefore excluded from Theorem~\ref{chebsaus}. Perhaps this could be the subject of further research.

We state here the result from Il'yashenko \cite{ilja}: 
If a sequence of limit cycles of an analytic vector field converges to a polycycle with saddle-node singular points, then one can select a semitransversal to this polycycle such that the displacement function is not equal to zero, but infinitely flat, for e.g. $f(x)\simeq e^{-\frac{1}{x}}$.
If we have a polycycle with only saddle singular points, then the displacement function cannot be infinitely flat.

\section{Appendix}\label{sec5}

In Appendix we put some observations concerning main results.

\begin{remark}[sublinearity in Theorem~\ref{gensaus}]\label{superlin}
The condition $m>1$ in the lower power condition in Theorem~\ref{gensaus} cannot be weakened. If we take, for example, the function $$f(x)=\frac{x}{-\log x},$$ it obviously satisfies all assumptions of Theorem \ref{gensaus}, except sublinearity: the lower power condition holds only for $m\leq 1$. If we compute $|A_{\varepsilon}(S^g(x_0))|$ for this  function (as it is computed in the proof of Theorem \ref{gensaus} in Section \ref{sec3}), we get that $\frac{|A_{\varepsilon}(S^g(x_0))|}{f^{-1}(\varepsilon)}$ tends to infinity, as $\varepsilon\to 0$, and therefore the conclusion $(\ref{gensausage})$ is not true.

On the other hand, for functions of the form $$f(x)=\frac{x^{1+\alpha}}{-\log x},\ \ \alpha>0,$$ which are obviously sublinear with $m=1+\alpha>1$, the explicit computation shows that $|A_{\varepsilon}(S^g(x_0))|\simeq f^{-1}(\varepsilon)$, as $\varepsilon\to 0$. 
\end{remark}

\begin{remark}[Upper power condition in Theorem~\ref{chebsaus}]\label{upperpower}
The upper power condition on $f$ is needed in Theorem~\ref{chebsaus}, for the if implication to hold, see Lemma~\ref{doubling}.i). As a counterexample, we can take the following Chebyshev scale
$$
\mathcal{I}=\{e^{-\frac{1}{x}},e^{-\frac{1}{2x}},e^{-\frac{1}{3x}},\ldots\},
$$
and,  $f(x)=e^{-\frac{1}{3x}}$, $g$=id$-f$ , which does not satisfy upper power condition.

Obviously, $D_0(f)(0)=D_1(f)(0)=0$ and $D_2(f)(0)>0$, so the multiplicity $\mu_0(g,\mathcal{G})\leq 2$. On the other side, $u_1^{-1}(\varepsilon)\simeq u_2^{-1}(\varepsilon)\simeq u_3^{-1}(\varepsilon)\simeq\ldots\simeq \frac{1}{-\log\varepsilon}$, therefore the critical Minkowski order $m(g,\mathcal{G})$ is infinite. In this case, we are not able to read the multiplicity neither from the critical Minkowski order nor from behavior of the length of the $\varepsilon-$neighborhood. 
\end{remark}

\bigskip

\begin{example}[Non-flat, non weakly comparable to powers function]\label{constr}
We construct a non infinitely flat function $f$ that does not satisfy $x (\log f)'(x)\leq M$ for any $M>0$, just to show that, for functions of interest, non-flatness is not  equivalent to weak comparability to powers.

The main idea is to bound the function from above and from below with $x^{\alpha+1}$ and $x^\alpha$, $\alpha>1$, therefore it cannot be infinitely flat.\\ Next we need to make sure that on some intervals approaching zero its logarithmic growth is faster than the logarithmic growth of $x^\alpha$.

We construct the function $f$ in logarithmic chart, i.e. we construct function $g(x)=\log f(x)$ on some segment $(0,d)$.\\ Let $h_1(x)=\log (x^{\alpha})=\alpha\log x$ and let $h_2(x)=\log (x^{\alpha+1})=(\alpha+1)\log x$. Let us take $x_1$ close to $x=0$. The segment $I_1$ connects the points $(x_1,h_1(x_1))$ and $(x_1/2, h_2(x_1/2))$. Now we choose point $x_2$ such that $h_1(x_2)<h_2(x_1/2)$ (to ensure that $f$ is increasing). We get segment $I_2$ by connecting $(x_1/2,h_2(x_1/2))$ and $(x_2,h_1(x_2))$. We repeat the procedure with $x_2$ instead of $x_1$ to get segments $I_3$ and $I_4$ and, inductively, we get the sequence $(x_n)$ tending to $0$, as $n\to\infty$, and the sequence of segments $(I_n)$ which are becoming perpendicular very quickly, see Figure \ref{fig}. 

The graph of our function $g$ will be the union of the segments $\bigcup_{n=1}^{\infty}I_n$, smoothened on edges. Obviously $f(x)=e^{g(x)}$ is bounded by $x^{\alpha+1}$ and $x^\alpha$. Nevertheless, if we take the sequence $(y_n)$ such that $x_n/2<y_n<x_n$, we compute
$$
g'(y_n)\cdot y_n=\frac{m\log x_n-(m+1)\log \frac{x_n}{2}}{x_n/2}\cdot y_n\ \simeq\ -\log x_n, \text{\ \ as $n\to\infty$},
$$ 
and, thus, for the sequence $(y_n)$ tending to $0$ it holds $g'(y_n)y_n\to \infty$, as $n\to\infty$, a contradiction to $x g'(x)\leq M$.
\end{example}

\begin{figure}
\centering
\includegraphics[scale=0.6]{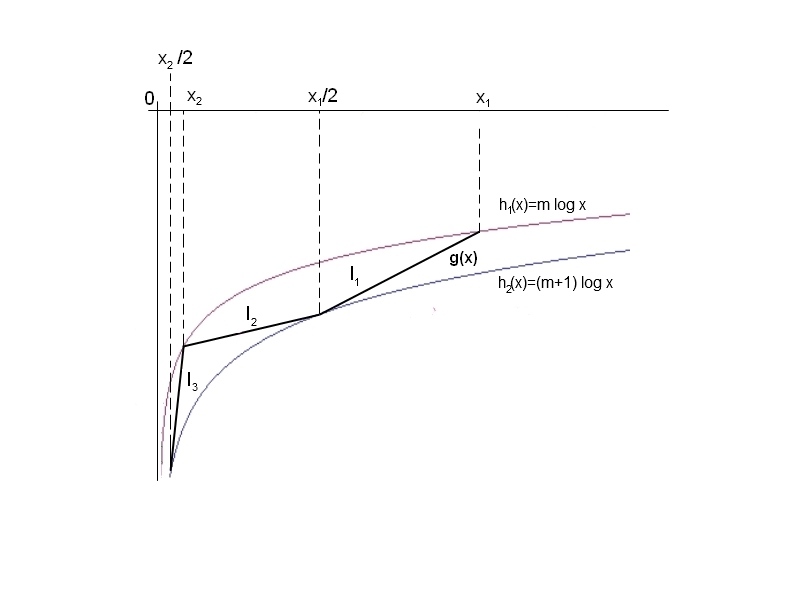}
\vspace{-2.5cm}
\caption{\scriptsize Function $g(x)=\log f(x)$ from Construction \ref{constr}.}
\label{fig}
\end{figure}

{\bf Acknowledgements:} {We would like to thank Darko \v Zubrini\'c for fruitful discussions.}

\end{document}